\documentclass[pdflatex,sn-mathphys,Numbered]{sn-jnl}

\usepackage[T1]{fontenc}
\usepackage{amssymb}
\usepackage{amsmath}
\usepackage[mathscr]{eucal}
\usepackage{graphicx}
\usepackage{array}
\usepackage{xcolor}
\usepackage{manyfoot}

\hypersetup{hidelinks,pdftitle={Constrained homomorphism orders},pdfauthor={Jiri Fiala, Jan Hubicka, Yangjing Long},pdfkeywords={graph homomorphism, partial order, graph core, duality, constrained homomorphism, locally constrained homomorphism}}
\providecommand{\burl}[1]{\url{#1}}

\newcolumntype{L}[1]{>{\raggedright\arraybackslash}p{#1}}
\makeatletter

\newcommand{\Graphs}[0]{\ensuremath{\mathsf{Graphs}}}
\newcommand{\ConnGraph}[0]{\ensuremath{\mathsf{ConnGraph}}}
\newcommand{\DiGraphs}[0]{\ensuremath{\mathsf{DiGraphs}}}
\newcommand{\DiCycle}[0]{\ensuremath{\mathsf{DiCycle}}}
\newcommand{\DiCycles}[0]{\ensuremath{\mathsf{DiCycles}}}
\newcommand{\Cycle}[0]{\ensuremath{\mathsf{Cycle}}}
\newcommand{\Cycles}[0]{\ensuremath{\mathsf{Cycles}}}
\newcommand{\UnorPath}[0]{\ensuremath{\mathsf{UnorPath}}}
\newcommand{\Path}[0]{\ensuremath{\mathsf{Path}}}

\newcommand{\Hom}{\mathop{\to}}

\newcommand{\SurHom}{\mathop{\xrightarrow{_S}}}
\newcommand{\VSurHom}{\mathop{\xrightarrow{_{VS}}}}

\newcommand{\FullHom}{\mathop{\xrightarrow{_F}}}
\newcommand{\MonoHom}{\mathop{\xrightarrow{_M}}}
\newcommand{\EmbedHom}{\mathop{\xrightarrow{_E}}}

\newcommand{\LocInHom}{\mathop{\xrightarrow{_{LI}}}}
\newcommand{\LocBiHom}{\mathop{\xrightarrow{_{LB}}}}

\newcommand{\Pfin}{\ensuremath{P_{\mathrm{fin}}}}
\newcommand{\subsetLeq}[1]{\mathop{\leq^{\mathrm{dom}}_#1}}

\newcommand{\SurLeq}{\mathop{\leq^S}}
\newcommand{\VSurLeq}{\mathop{\leq^{VS}}}
\newcommand{\ESurLeq}{\mathop{\leq^{ES}}}
\newcommand{\LocSurLeq}{\mathop{\leq^{LS}}}
\newcommand{\LocInLeq}{\mathop{\leq^{LI}}}
\newcommand{\LocInL}{\mathop{<^{LI}}}
\newcommand{\LocBiLeq}{\mathop{\leq^{LB}}}
\newcommand{\EmbedLeq}{\mathop{\leq^E}}

\newcommand{\MonoLeq}{\mathop{\leq^M}}
\newcommand{\FullLeq}{\mathop{\leq^F}}

\newcommand{\K}{\mathop{\mathcal{K}}}

\let\Homeq\sim

\newcommand{\FullHomeq}{\mathop{\overset{_F}{\Homeq}}}
\newcommand{\SurHomeq}{\mathop{\overset{_S}{\Homeq}}}
\newcommand{\Matrixeq}{\mathop{\overset{_M}{\Homeq}}}

\newcommand{\drm}{\mathrm{drm}}

\let\notto\nrightarrow

\newtheorem{defn}{Definition}[section]

\newtheorem{thm}{Theorem}[section]

\newtheorem{corollary}[thm]{Corollary}
\newtheorem{lem}[thm]{Lemma}
\newtheorem{prop}[thm]{Proposition}

\DeclareRobustCommand{\qed}{%
  \ifmmode \mathqed
  \else
    \leavevmode\unskip\penalty9999 \hbox{}\nobreak\hfill
    \quad\hbox{$\square$}%
  \fi
}
\let\QED@stack\@empty
\let\qed@elt\relax
\newcommand{\pushQED}[1]{%
  \toks@{\qed@elt{#1}}\@temptokena\expandafter{\QED@stack}%
  \xdef\QED@stack{\the\toks@\the\@temptokena}%
}
\newcommand{\popQED}{%
  \begingroup\let\qed@elt\popQED@elt \QED@stack\relax\relax\endgroup
}
\def\popQED@elt#1#2\relax{#1\gdef\QED@stack{#2}}
\newcommand{\qedhere}{%
  \begingroup \let\mathqed\math@qedhere
    \let\qed@elt\setQED@elt \QED@stack\relax\relax \endgroup
}

\providecommand{\proofname}{Proof}
\newenvironment{proof}[1][\proofname]{\par
  \pushQED{\qed}%
  \normalfont \topsep6\p@\@plus6\p@\relax
  \trivlist
  \item[\hskip\labelsep \bf {#1\ignorespaces.}]\ignorespaces
}{%
\popQED\endtrivlist
\par
}
\newenvironment{proof*}[1][\proofname]{\par
  \normalfont \topsep6\p@\@plus6\p@\relax
  \trivlist
  \item[\hskip\labelsep \bf {#1\ignorespaces.}]\ignorespaces
}{%
\endtrivlist
\par
}

\providecommand{\remarkname}{Remark}

\makeatother

\articletype{Research Article}

\title[Constrained homomorphism orders]{Constrained homomorphism orders}

\author[1]{\fnm{Ji\v{r}\'\i{}} \sur{Fiala}}
\email{fiala@kam.mff.cuni.cz}

\author[1]{\fnm{Jan} \sur{Hubi\v{c}ka}}
\email{hubicka@kam.mff.cuni.cz}

\author*[2]{\fnm{Yangjing} \sur{Long}}
\email{yangjing@ccnu.edu.cn}

\affil[1]{\orgdiv{Department of Applied Mathematics}, \orgname{Charles University}, \orgaddress{\city{Prague}, \country{Czech Republic}}}
\affil[2]{\orgdiv{School of Mathematics and Statistics}, \orgname{Central China Normal University}, \orgaddress{\city{Wuhan}, \country{China}}}

\abstract{We study partial orders induced by constrained variants of finite graph homomorphisms: monomorphisms, embeddings, full homomorphisms, vertex-surjective, edge-surjective and surjective homomorphisms, and locally injective, locally surjective and locally bijective homomorphisms. For each order we ask for analogues of the standard structural properties of the graph homomorphism order: canonical cores, past- or future-finiteness, universality, gaps and finite dualities. The comparison shows which phenomena are specific to ordinary homomorphisms and which are consequences of simpler order-theoretic mechanisms. We identify cores for full and surjective homomorphisms, relate full-homomorphism cores to point-determining graphs, characterize gaps in the full homomorphism order, and give finite obstruction bounds for several one-sided finite orders. We also analyze locally constrained homomorphisms on connected graphs. In particular, locally injective homomorphisms have all connected graphs as cores, admit infinite-chain density under natural degree-refinement assumptions, have explicit gap witnesses, and are universal already on finite connected bipartite subcubic cactus graphs. The paper reorganizes and extends several earlier arguments into a single framework for constrained homomorphism orders.}

\keywords{graph homomorphism; partial order; graph core; duality; constrained homomorphism; locally constrained homomorphism}

\begin{document}
\maketitle

\section{Introduction}
\label{sec:introduction}

Graph homomorphisms form one of the central organizing notions in structural graph theory.  They generalize graph colorings, give rise to constraint satisfaction problems, and produce a rich partial order after graphs that are homomorphically equivalent are identified by their cores; see the monograph of Hell and Ne\v{s}et\v{r}il~\cite{Hell2004}.  The ordinary homomorphism order is exceptional in several ways: it is universal, its gaps and dualities are well understood, and many of its suborders retain considerable complexity.

The purpose of this paper is to compare this classical picture with the orders obtained when the homomorphisms are required to satisfy additional local or global constraints.  The constraints considered here include injectivity, fullness, vertex and edge surjectivity, and local injectivity, local surjectivity and local bijectivity.  The guiding question is the following: which order-theoretic phenomena survive under these constraints, and which ones are special to the ordinary homomorphism order?

\subsection{Graph homomorphism orders}

A \emph{directed graph} $G$ is a pair $G=(V_G,E_G)$ with $E_G\subseteq V_G\times V_G$.  The class of all finite directed graphs is denoted by $\DiGraphs$.  An \emph{undirected graph} is viewed as a directed graph satisfying $(u,v)\in E_G$ if and only if $(v,u)\in E_G$.  Unless explicitly stated otherwise, loops are allowed.  The class of all finite undirected graphs is denoted by $\Graphs$.  The locally constrained part of the paper is the main exception: there, unless explicitly stated otherwise, we work with finite, undirected, loopless graphs and with the usual open neighborhoods.

For a directed graph $G$ and a vertex $u\in V_G$, let $N_G(u)$ be the open neighborhood of $u$.  For directed graphs $G$ and $H$, a \emph{homomorphism} $f:G\to H$ is a map $f:V_G\to V_H$ such that $(u,v)\in E_G$ implies $(f(u),f(v))\in E_H$.  We write $G\Hom H$ when such a map exists.  If $G\Hom H$ and $H\Hom G$, then $G$ and $H$ are \emph{homomorphically equivalent}, written $G\Homeq H$.

The relation $\Hom$ is reflexive and transitive, so it is a quasi-order on $\DiGraphs$.  A finite directed graph is a \emph{core} if it has the smallest possible number of vertices in its homomorphic equivalence class.  Every homomorphic equivalence class contains a unique core up to isomorphism, and this core is an induced subgraph of every graph in the class~\cite{Hell2004}.  We obtain the homomorphism order by taking cores as canonical representatives.  We use the same convention for all constrained homomorphism orders below: after passing to the appropriate cores, the same symbol denotes both the quasi-order and the induced partial order.

A partial order $(P,\leq_P)$ is \emph{universal} if every countable partial order embeds into it.  The homomorphism orders of finite graphs and finite directed graphs are universal; this was first proved in a stronger categorical form by Pultr and Trnkov\'a~\cite{pultr1980combinatorial}.  Several restricted graph classes and related structures are also universal, including set systems~\cite{Nesetril2000}, oriented trees~\cite{Hubicka2005}, oriented paths~\cite{Hubicka2004}, and labeled posets~\cite{Lehtonen2008}.

We shall also use the language of gaps and dualities.  A \emph{gap} in a partial order is a pair $(u,v)$ with $u<v$ and no element strictly between $u$ and $v$.  A partial order without gaps is \emph{dense}.  For a class $\K$ of directed graphs, a pair $(F,D)$ is a \emph{simple duality pair} if, for every $G\in\K$,
$$
G\notto D \quad\hbox{if and only if}\quad F\to G.
$$
More generally, for finite sets of directed graphs $\mathcal{F}$ and $\mathcal{D}$, the pair $(\mathcal{F},\mathcal{D})$ is a \emph{generalized finite duality pair} in $\K$ if, for every $G\in\K$, some $F\in\mathcal{F}$ maps to $G$ precisely when $G$ maps to no member of $\mathcal{D}$~\cite{Foniok2008}.  In the ordinary homomorphism order, gaps and finite dualities are closely related~\cite{Nesetril2000}.

\subsection{Constrained homomorphisms}

We call the following variants of graph homomorphisms \emph{constrained homomorphisms}.

\begin{defn}
A homomorphism $f:G\to H$ is
\begin{itemize}
\item[(M)] a \emph{monomorphism}, if it is injective;
\item[(F)] a \emph{full homomorphism}, if $(f(u),f(v))\in E_H$ implies $(u,v)\in E_G$;
\item[(E)] an \emph{embedding}, if it is both a monomorphism and a full homomorphism;
\item[(VS)] a \emph{vertex-surjective homomorphism}, if $f(V_G)=V_H$;
\item[(ES)] an \emph{edge-surjective homomorphism}, if for every $(u,v)\in E_H$ there is $(u',v')\in E_G$ such that $f(u')=u$ and $f(v')=v$;
\item[(S)] a \emph{surjective homomorphism}, if it is both vertex-surjective and edge-surjective;
\item[(LI)] \emph{locally injective}, if for every vertex $v$ the restriction of $f$ from $N_G(v)$ to $N_H(f(v))$ is injective;
\item[(LS)] \emph{locally surjective}, if for every vertex $v$ the restriction of $f$ from $N_G(v)$ to $N_H(f(v))$ is surjective;
\item[(LB)] \emph{locally bijective}, if the restriction of $f$ from $N_G(v)$ to $N_H(f(v))$ is bijective for every vertex $v$.
\end{itemize}
\end{defn}

We use the symbols attached to these constraints to denote the corresponding existence relations and orders.  Thus $G\FullHom H$ denotes the existence of a full homomorphism from $G$ to $H$, and $G\FullLeq H$ denotes the corresponding order relation on $F$-cores.  The notation is analogous for $\MonoLeq$, $\EmbedLeq$, $\SurLeq$, $\VSurLeq$, $\ESurLeq$, $\LocInLeq$, $\LocSurLeq$ and $\LocBiLeq$.

Recent work confirms that these constrained variants remain active.  Full homomorphisms are studied through obstruction sets for full $H$-colouring~\cite{GuzmanPro2023FullPathsCycles}; surjective homomorphisms and compactions through Sur-CSP, reflexive digraph targets and counting complexity~\cite{Larose2019SurjectiveReflexive,Focke2019CountingSurjective,Chen2024SurjectiveCSP}; locally constrained homomorphisms through algorithmic frameworks and list locally surjective variants~\cite{Bulteau2024Algorithmic,Dvorak2022ListLS}; and graph covers through topological and complexity questions~\cite{KratochvilNedela2025CoversSemicovers}.  These developments are mainly algorithmic, obstruction-theoretic or topological; here we focus on the induced partial orders.

\subsection{Main results and organization}

The paper is organized around the comparison in Table~\ref{tab:main-results}.  Section~\ref{sec:universality} develops the order-theoretic tools used throughout the paper.  Sections~\ref{sec:embedding}, \ref{sec:fullhomo} and \ref{sec:surjhomo} treat the global constraints.  Section~\ref{sec:locally-constrained} treats locally constrained homomorphisms, where the locally injective order on connected graphs is the most similar to the ordinary homomorphism order.

\begin{table}[t]
\caption{Summary of the constrained homomorphism orders studied here.  All graphs are finite unless a connected or directed class is specified.}
\label{tab:main-results}
{\small
\setlength{\tabcolsep}{3pt}
\renewcommand{\arraystretch}{1.12}
\begin{tabular}{L{0.17\textwidth}L{0.24\textwidth}L{0.25\textwidth}L{0.22\textwidth}}
\hline
Order & Cores & Finiteness / universality & Gaps and dualities \\
\hline
M and E & all finite graphs are canonical representatives & past-finite; past-finite-universal on cycles & generalized finite dualities exist with $n+1$-vertex obstructions \\
F & exactly the point-determining graphs & coincides with the embedding order on cores; past-finite-universal on cycles & gaps are characterized; finite dualities have $n+1$-vertex obstructions \\
VS, ES and S & all graphs for VS and S; ES-cores are obtained by deleting isolated vertices & future-finite; future-finite-universal on directed cycles & gaps and finite dualities are characterized \\
LS and LB on connected graphs & every connected graph is a core & future-finite; universal after passing to disconnected unions inside nontrivial degree-refinement classes & many gaps; no generalized finite duality pairs on connected graphs \\
LI on connected graphs & every connected graph is a core & universal already on connected bipartite subcubic cacti; infinite-chain density in natural intervals & explicit gap witnesses; finite dualities occur exactly for finite tree target families \\
\hline
\end{tabular}}
\end{table}

Several ingredients of the paper appeared earlier in extended-abstract form: in particular the universality argument for graph homomorphisms~\cite{Fiala2015Universality} and the gap theorem for full homomorphisms~\cite{Fiala2017GapsFull}.  The present journal version expands these conference versions by giving full proofs in a common notation, by adding the unified treatment of constrained homomorphism orders, by isolating the one-sided-finiteness mechanism behind several duality results, and by making the comparison between the individual orders explicit in Table~\ref{tab:main-results}.

\section{Order-theoretic tools}
\label{sec:universality}

\subsection{Universal partial orders}

In this section we review a basic technique for proving universality of homomorphism orders.
We developed this technique originally to prove universality for locally injective homomorphisms (Theorem~\ref{thm:lochomouniv});
it also appears in~\cite{Fialab,FialaFractal}.

An {\em embedding} of a partial order $(Q,\leq_Q)$ into $(P,\leq_P)$ is a mapping $e:Q\to P$ satisfying
$x\leq_Q y$ if and only if $e(x)\leq_P e(y)$. In such a case we also say that $(Q,\leq_Q)$ is a {\em suborder} of $(P,\leq_P)$.

For a given partial order $(P,\leq)$, the {\em down-set} $\downarrow x$ is $\{y\in P \mid y\leq x\}$. 
Similarly, the {\em up-set} is $\uparrow x = \{y\in P \mid x\leq y\}$.

Any finite partial order $(P,\leq)$ can be represented by finite sets ordered by inclusion, e.g.
when $x$ is represented by $\downarrow x$. 
This is a valid embedding, because $\downarrow x \subseteq\ \downarrow y$ if and only if $x\leq y$.

Without loss of generality we may assume that $P$ is a subset of some fixed countable set $A$, e.g. $\mathbb N$. 
Consequently, the partial order formed by the system $\Pfin(A)$ of all finite subsets of $A$ 
ordered by inclusion contains any finite partial order as a suborder. Such orders are called {\em finite-universal}.
We reserve the term {\em universal} for orders that contain every countable partial order as a suborder.

We briefly sketch a standard construction of a universal partial order, since the same technique is used later for homomorphism orders. 
The universal partial order is built in two steps. 
For these we need further terminology: An order is {\em past-finite}, if every down-set is finite. 
An order is {\em past-finite-universal} if it contains every past-finite order. Analogously, 
{\em future-finite} and {\em future-finite-universal} orders are defined w.r.t. finiteness of up-sets.

\paragraph{1.}
Observe that the mapping $e(x) = \downarrow x$ is also an embedding $e:(P,\leq)\to(\Pfin(A),\subseteq)$ in the case when $(P,\leq)$ 
is past-finite and $P\subseteq A$.
Since a past-finite partial order turns to be future-finite when the direction of inequalities is reversed, we get:

\begin{corollary}
\label{cor:pastfuturefiniteuniv}
For any countably infinite set $A$ it holds that
\begin{itemize}
\item[(i)] the order $(\Pfin(A),\subseteq)$ is past-finite-universal, and
\item[(ii)] the order $(\Pfin(A),\supseteq)$ is future-finite-universal.
\end{itemize}
\end{corollary}

\paragraph{2.}
For a given partial order $(Q,\leq)$ we construct the {\em subset order}, $(\Pfin(Q),\subsetLeq{Q})$, where 
$$X\subsetLeq{Q} Y\iff \hbox{ for every }x\in X \hbox{ there exists }y\in Y \hbox{ such that }x\leq y.$$
In such case we say that $X$ {\em is dominated by} $Y$.

Observe that for any order $(Q,\leq)$ the relation $\subsetLeq{Q}$ is transitive and reflexive, i.e. it yields a quasi-order on $\Pfin(Q)$.
We choose a unique representative for every equivalence class to get a partial order. In particular, we may represent $X\subseteq Q$ 
by the antichain of the maximum elements of $X$.

The subset orders can be layered to get a universal order.
\begin{thm}[Universality of the subset order]

\label{thm:univ}
For every future-finite-universal partial order $(F,\leq_F)$ it holds that $(\Pfin(F),\subsetLeq{F})$ is universal.
\end{thm}

\begin{proof}
Let any countable partial order $(P,\leq_P)$ be given. Without loss of generality we may assume that $P\subseteq \mathbb{N}$. This
way we enforce a linear order $\leq$ on the elements of $P$. The order $\leq$ is unrelated to the partial order $\leq_P$.

We decompose $(P,\leq_P)$ into two parts: 
\begin{enumerate}
\item The \emph{forward order} $\leq_f$, where $x\leq_f y$ if and only if $x\leq_P y$ and $x \leq y$, 
\item the \emph{backward order} $\leq_b$, where $x\leq_b y$ if and only if $x\leq_P y$ and $x \geq y$.
\end{enumerate}

For every $x\in P$ both sets $\{y\mid y\leq_f x\}$ and $\{y\mid x\leq_b y\}$ are finite.
In other words $(P,\leq_f)$ is past-finite and $(P,\leq_b)$ is future-finite.

Since $(F,\leq_F)$ is future-finite-universal, 
there is an embedding $e: (P,\mathop{\leq_b}) \to (F,\leq_F)$. 
For every $x\in P$ we now define: 
$$g(x)=\{e(y)\mid y\leq_f x\}.$$
We argue that $g$ is an embedding of $(P,\leq_P)$ into $(\Pfin(F),\subsetLeq{F})$.

First we show that $g(x)\subsetLeq{F} g(y)$ implies $x\leq_P y$.
From the definition of $\subsetLeq{F}$ follows that there is
$w\in P$, $e(w)\in g(y)$, such that $e(x)\leq_F e(w)$. By the definition of $g$, $e(w)\in g(y)$ if and only if $w\leq_f y$. 
By the definition of $e$, $e(x)\leq_F e(w)$ if and only if $x\leq_b w$.
Together we get that $x\leq_b w\leq_f y$, which implies $x\leq_P w\leq_P y$ and in particular $x\leq_P y$.

To show that $x\leq_P y$ implies $g(x)\subsetLeq{F} g(y)$ we consider two cases.
\begin{enumerate}
\item When $x\leq y$ then $g(x)\subseteq g(y)$ and thus also $g(x)\subsetLeq{F} g(y)$.
\item Now assume that $x>y$. For every $w\in P$, such that $e(w)\in g(x)$ it holds that $w\leq_f x\leq_P y$ by the construction of $g(x)$.
Now either $w\leq y$, which implies that $e(w)\in g(y)$, or $w\geq y$ which implies $w\leq_b y$ and consequently  $e(w)\leq_F e(y)$. 
As some case applies on each $e(w)\in g(x)$, it follows that $g(x)$ is dominated by $g(y)$.
\end{enumerate}
\end{proof}

The same construction is illustrated in Figure~\ref{fig:sampleposet}.

\begin{figure}[t]
\centering
\includegraphics[width=0.88\textwidth]{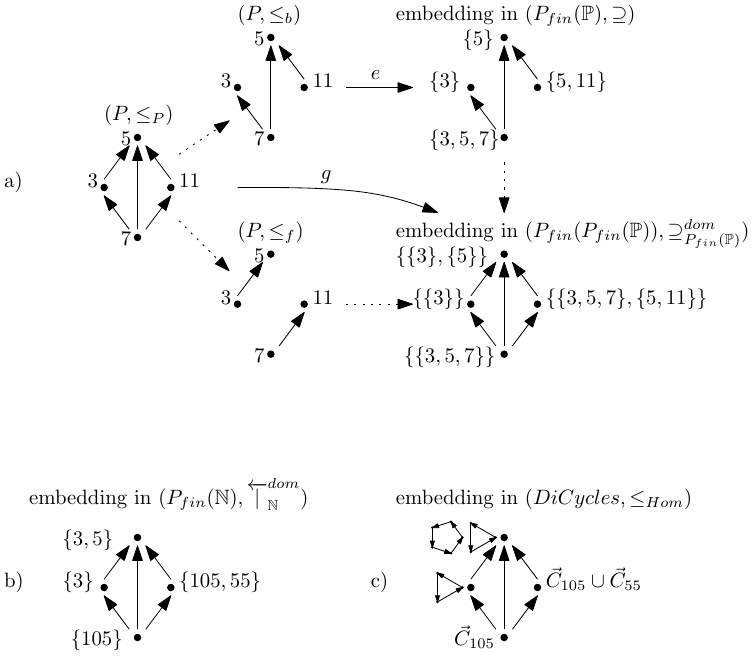}
\caption{Three representations of the same finite partial order: by Theorem~\ref{thm:univ}, by the subset divisibility order, and by homomorphisms of disjoint unions of directed cycles.}
\label{fig:sampleposet}
\end{figure}

The embedding $g$ constructed in the proof of Theorem \ref{thm:univ} has the 
property that $g(x)$ depends only on elements $y<x$. Such embeddings are known 
as {\em on-line embeddings} because they can be constructed inductively without 
a priori knowledge of the whole partial order~\cite{Hubicka2011}. All embeddings 
shown later in our paper are on-line as well.

This construction goes back to Hedrl\'in~\cite{Hedrlin1969} and was later used for set systems, oriented paths and other homomorphism orders~\cite{Nesetril2000,Hubicka2004,Hubicka2011}.  We use the layered formulation in terms of future-finite-universal and past-finite-universal orders; a less constructive existence proof for past-finite-universal structures appears in~\cite{Droste2005}.

We apply Theorem~\ref{thm:univ} on a particular order derived from the divisibility order. 
We write the divisibility relation $a|b$, meaning "$a$ divides $b$", in the reverse order as $b \overleftarrow| a$, 
meaning "$b$ is divided by $a$".

By Corollary~\ref{cor:pastfuturefiniteuniv} we see that a particular example of a past-finite-universal  is $(\Pfin(\mathbb{P}),\subseteq)$, where $\mathbb{P}$ is the class of all odd prime numbers, and hence $(\Pfin(\mathbb{P}),\supseteq)$ is future-finite-universal.

Since, for $X,Y\in \Pfin(\mathbb{P})$, $X \subseteq Y$ if and only if $\prod X | \prod Y$, we immediately obtain special embeddings of the subset orders by divisibility (see Figure~\ref{fig:sampleposet} b) as: 

\begin{prop}\label{prop:setdiviuniv}
\begin{itemize}
\item[a)] The order $(\mathbb{N},|)$ is past-finite-universal,
\item[b)] the order $(\mathbb{N},\overleftarrow|)$ is future-finite-universal,
\item[c)] the \emph{subset divisibility order} $(\Pfin({\mathbb N}),\overleftarrow|^{dom}_{\mathbb N})$ is universal.
\end{itemize}
\end{prop}

\subsection{A universality argument for homomorphism orders}

An {\em oriented path} $P$ is a orientation of an undirected path, formally $V_P=\{v_0,v_1,\ldots,v_n\}$ and for every $i=1,2,\ldots,n$ either $(v_{i-1},v_i)\in E_P$ or $(v_i,v_{i-1})\in E_P$, but not both, and there are no other edges. 
We denote by $\Path$ the class of all finite oriented paths.

Let further $C_k$ stand for the unoriented cycle on $k$ vertices; $\Cycle$ for the class of graphs formed by all $C_k$, $k\geq 3$;
and $\Cycles$ for the class of graphs formed by disjoint union of finitely many graphs in $\Cycle$.

Analogously, $\overrightarrow{C}_k$ is the directed cycle on $k$ vertices with edges oriented in the same direction;  
$\DiCycle$ is the class of directed graphs formed by all $\overrightarrow{C}_k$, $k\geq 3$; and $\DiCycles$ is the class of directed graphs formed by disjoint union of finitely many graphs in $\DiCycle$.

It is a classical result that the partial order $(\Graphs,\leq)$ is a universal partial order. In fact, graph homomorphisms are universal also in the categorical sense \cite{pultr1980combinatorial}.  It has recently been shown that the universality of the homomorphism order is a relatively persistent property, and that the homomorphism order remains universal even for very restricted classes of (directed) graphs. 
Among others:

\begin{thm}[Hubi\v{c}ka, Ne\v{s}et\v{r}il \cite{Hubicka2004,Hubicka2011}]
\label{thm:paths}
The partial order $(\Path,\leq)$ is universal.
\end{thm}

Consequently the homomorphism order is universal in many naturally defined classes of (directed) graphs, such as 3-colorable graphs, planar graphs, series-parallel graphs, etc. As a pleasant surprise, we get an elementary argument for a particular class of digraphs:

\begin{thm}[Fiala, Hubi\v{c}ka, Ne\v{s}et\v{r}il \cite{Fialab}]
\label{thm:cycles}
The partial order $(\DiCycles,\leq)$ is universal.
\end{thm}

\begin{proof}
As $\overrightarrow{C}_k\to \overrightarrow{C}_l$ if and only if $k \overleftarrow| l$, we get the conclusion directly from 
Proposition~\ref{prop:setdiviuniv}. See also Figure~\ref{fig:sampleposet} c).
\end{proof}

Note that our universality argument for $(\DiCycles,\leq)$ uses a ``rolling'' principle instead of ``flipping'' used before by others. In addition, the ``rolling'' approach can also be applied in the context of locally constrained homomorphisms, while these homomorphisms do not allow any sort of ``flipping'' operation necessary for non-trivial path homomorphisms.

Theorem~\ref{thm:cycles} has similar applications as Theorem \ref{thm:paths}~\cite{Nesetril2006,Lehtonen2008,Lehtonen2010,Kwuida2011,Samal2013} as well as new applications like the universality on homomorphisms of line graphs~\cite{Fialab}.

\subsection{Gaps and dualities in one-sided finite orders}
\label{sec:pastfinite}

The finiteness of down-sets or up-sets is a characteristic property of many constrained homomorphism orders. 
We now make some elementary observations about past-finite and future-finite partial orders.  
Both kinds of orders are also locally-finite~\cite{Britz2001,Droste2005}, where 
{\em locally-finiteness} of $(P,\leq_P)$ means that for every $x\leq_P y$, the interval $x\leq_P w\leq_P y$ 
is finite. 

The following claims are straightforward:

\begin{prop}[Gaps in locally-finite orders]
\label{prop:futurefinitegraps}
No locally-finite partial order is dense.
Moreover, for every $u\in P$, that is not maximal in $(P,\leq)$, there is an element $v\in P$ such that $(u,v)$ is a gap.
Analogously, for any non-minimal $u$, there exists $v$ such that $(v,u)$ is a gap.
\end{prop}

\paragraph{Dualities in past-finite and future-finite partial orders}
The concept of duality was studied for some past-finite and future-finite partial orders, see \cite{Hell2013,Ball2010,Feder2008,Xie2006}. 

A generalized finite duality pair $(\mathcal{F},\mathcal{D})$ in $(P,\leq_P)$ is exactly a finite partition of the order into principal up-sets generated by $\mathcal{F}$ and principal down-sets generated by $\mathcal{D}$:
$$\bigcup_{f\in \mathcal{F}}\uparrow f=P\setminus \bigcup_{d\in\mathcal{D}}\downarrow d.$$
For graph homomorphisms, such a pair is a splitting antichain in the order~\cite{Ahlswede1995,Foniok2008}.

When $(P,\leq)$ is future-finite, then for any finite $\mathcal{F}$, the union $\bigcup_{f\in \mathcal{F}}\uparrow f$ 
is always finite. Such $\mathcal{F}$ may be accompanied by $\mathcal{D}$ to form a duality pair $(\mathcal{F},\mathcal{D})$ if and only if 
the above equation is satisfied. Necessarily, $P\setminus \bigcup_{d\in\mathcal{D}}\downarrow d$ is finite, 
which itself happens rather rarely. 

We thus ask for a characterization of sets $\mathcal{F}$ that form a generalized finite duality pair $(\mathcal{F},\mathcal{D})$ for some $\mathcal{D}$. We show that the existence of finite duality pairs is closely related to the existence of gaps. 
Though an analogous relationship holds for the homomorphism order too, it is of different nature here.

\begin{prop}[Dualities in future-finite orders]
\label{prop:futurefiniteduals}
Let $(P,\leq)$ be a future-finite partial order and let $\mathcal{F}\subseteq P$ be finite.  Put
\[
X=\bigcup_{f\in \mathcal{F}}\uparrow f .
\]
There is a finite set $\mathcal{D}\subseteq P$ such that
\[
X=P\setminus \bigcup_{d\in\mathcal{D}}\downarrow d
\]
if and only if the down-set $P\setminus X$ has only finitely many maximal elements.  In that case one may take
\[
\mathcal{D}=\max(P\setminus X),
\]
and this is the unique irredundant antichain of dual elements.
\end{prop}

\begin{figure}
\centerline{\includegraphics{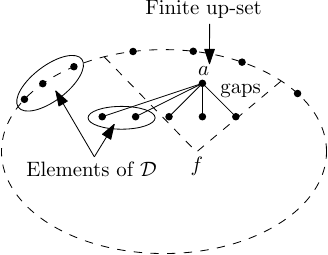}}
\caption{Dualities in a future-finite partial order.}
\label{fig:finiteduality}
\end{figure}

The situation is schematically depicted in Figure~\ref{fig:finiteduality}.
\begin{proof}
The set $X$ is an up-set, and hence $B=P\setminus X$ is a down-set.  Since the order is future-finite, every element of $B$ lies below a maximal element of $B$: indeed, one can choose a maximal element of the finite set $B\cap\uparrow b$.

If $\max B$ is finite, then $B=\bigcup_{d\in\max B}\downarrow d$, and therefore the displayed duality equation holds with $\mathcal{D}=\max B$.

Conversely, suppose that $B=\bigcup_{d\in\mathcal{D}}\downarrow d$ for a finite set $\mathcal{D}$.  No element of $\mathcal{D}$ lies in $X$, since otherwise that element would belong to both sides of the partition.  Hence $\mathcal{D}\subseteq B$.  If $m$ is maximal in $B$, then $m\leq d$ for some $d\in\mathcal{D}\subseteq B$, and maximality gives $m=d$.  Thus $\max B\subseteq\mathcal{D}$, so $\max B$ is finite.  The equality $B=\bigcup_{m\in\max B}\downarrow m$ also shows that any irredundant antichain of dual elements must be exactly $\max B$.
\end{proof}

Equivalently, the dual elements form the upper boundary of the complement of $X$.  A maximal element of $P\setminus X$ is either maximal in $P$ itself or the lower endpoint of a gap whose upper endpoint belongs to $X$.  This gap interpretation is the form used in the applications below.

We use the above proposition also for past-finite orders, which rephrased reads as follows.

\begin{corollary}[Dualities in past-finite orders]
\label{cor:pastfiniteduals}
Let $(P,\leq)$ be past-finite and let $\mathcal{D}\subseteq P$ be finite.  Put
\[
Y=\bigcup_{d\in \mathcal{D}}\downarrow d .
\]
There is a finite set $\mathcal{F}\subseteq P$ such that
\[
\bigcup_{f\in\mathcal{F}}\uparrow f=P\setminus Y
\]
if and only if $P\setminus Y$ has only finitely many minimal elements.  In that case one may take
\[
\mathcal{F}=\min(P\setminus Y),
\]
and this is the unique irredundant antichain of left-hand obstructions.  A minimal element of $P\setminus Y$ is either minimal in $P$ or the upper endpoint of a gap whose lower endpoint belongs to $Y$.
\end{corollary}

We say that $(P,\leq)$ has {\em all right generalized finite dualities} if
for every finite $\mathcal{F}\subseteq P$ there is a finite $\mathcal{D}\subseteq
P$ such that $(\mathcal{F},\mathcal{D})$ is a generalized finite duality pair.

We remark that the connection between dualities and computational complexity 
does not hold in the case of past-finite and future-finite partial orders.
Since the set of objects described by dualities is finite, 
it is easy to write a recognition algorithm for a fixed generalized finite duality pair.

\section{Monomorphism and embedding orders}
\label{sec:embedding}

The essential property behind the results in this section is that the underlying mapping is injective. 
For brevity we state claims for $(\Graphs,\MonoLeq)$, however they are also valid for 
$(\Graphs,\allowbreak \EmbedLeq)$, $(\DiGraphs,\allowbreak \MonoLeq)$ and $(\DiGraphs,\EmbedLeq)$, and also 
for analogous subclasses like $\DiCycles$, etc.

Since any monomorphism is injective, the order $(\Graphs,\MonoLeq)$ is past-finite.

As cycles of distinct lengths are incomparable in $\MonoLeq$, we get by Corollary~\ref{cor:pastfuturefiniteuniv} (i)  
that the order $(\Cycles,\MonoLeq)$ is past-finite-universal.

Dualities in this context often exist:

\begin{prop}[Dualities]
\label{prop:embeddual}
For every finite set of graphs $\mathcal{D}$ there is a finite set
$\mathcal{F}$ such that $(\mathcal{F},\mathcal{D})$ is a generalized finite duality pair in $(\Graphs,\MonoLeq)$.
Moreover, if $n=\max_{D\in\mathcal{D}} |V_D|$, then $\mathcal{F}$ may be chosen among graphs on at most $n+1$ vertices.
\end{prop}

\begin{proof}
Denote by $n$ the maximal number of vertices among graphs in $\mathcal{D}$.
First construct the set $\mathcal{F}'$ as the set of all graphs $F$ such that $|V_F|\leq
n+1$ and $F\MonoHom D$ for no $D\in \mathcal {D}$. Then choose $\mathcal{F}$ as
the set of all mutually non-isomorphic minimal elements of $(\mathcal{F}',\MonoLeq)$.

We claim that $(\mathcal{F},\mathcal{D})$ is a generalized finite duality pair.
The set $\mathcal{F}$ is finite. 
Consider any graph $G$ such that there is no $D\in \mathcal{D}$, $G\MonoHom D$. There are two
cases:
\begin{enumerate}
  \item If $|V_G|\leq n+1$ then $G\in \mathcal{F}'$ and thus also there is $F\in \mathcal{F}$,
$F\MonoHom G$.
  \item If $|V_G|>n+1$ then any subgraph $G'$ of $G$ on $n+1$ vertices yields 
$G'\in \mathcal{F'}$, $G'\MonoHom G$ which can be as in the previous case accompanied with a suitable $F\in \mathcal{F}$,
$F\MonoHom G'\MonoHom G$.
\end{enumerate}
\end{proof}

\section{Full homomorphisms}
\label{sec:fullhomo}

As a first non-trivial case we consider full homomorphisms, that is, edge- and non-edge-preserving mappings. The full homomorphism order is well established:
The dualities in the full homomorphism order were studied by Ball, Ne\v{s}et\v{r}il and Pultr in
\cite{Ball2007,Ball2010}.  The problem of the existence of a full homomorphism from $G$ to $H$ is also known as
{\em full $H$-coloring} problem and in this language it was independently studied by Feder and Hell
in \cite{Feder2008} and, more recently, by Hell and Hern{\'a}ndez-Cruz~\cite{Hell2013}. In both cases the motivation was
the correspondence of full homomorphisms to a certain matrix partition problem, see~\cite{Feder2008}.  Recent work continues this obstruction-set viewpoint, for example for full homomorphisms to paths and cycles~\cite{GuzmanPro2023FullPathsCycles}.

We closely relate the full homomorphism order to the embedding order.
First we show that cores in full homomorphisms correspond to point-determining graphs which have been studied in the 1970s by Sumner~\cite{Sumner1973} (c.f. Feder and Hell~\cite{Feder2008}).
Consequently, we show that the cores are ordered by embeddings.

Though we will state our results on full homomorphisms for undirected graphs, they indeed hold for directed graphs with almost identical proofs. The only difference is that in a directed graph $G$, one shall consider 
the out-neighborhood of a vertex 
as well as the {\em in-neighborhood}, both defined intuitively.
Two vertices in a directed graph have the same neighborhoods if both the in- and out-neighborhoods match.  

{\em Point-determining graphs} (also known as {\em mating-type graphs}, {\em mating graphs}, {\em M-graph} or {\em thin graphs}) are graphs in which no two vertices have the same neighborhoods.  If we start with any graph $G$, and gradually merge vertices with the same neighborhoods, we obtain a point-determining graph, denoted by $G_{\mathrm{pd}}$.

It is easy to observe that $G_{\mathrm{pd}}$ is always an induced subgraph of $G$. Moreover, for every directed graph $G$ it holds that $G_{\mathrm{pd}}\FullHom G\FullHom G_{\mathrm{pd}}$ and thus $G\FullHomeq G_{\mathrm{pd}}$.  This motivates the following proposition:
\begin{prop}[Cores]
\label{prop:F-core}
A graph $G$ is an F-core if and only if it is point-determining.
\end{prop}
\begin{proof}
Recall that $G$ is an F-core if it is minimal (in the number of vertices) within its
equivalence class of $\FullHomeq$. If $G$ is an F-core, $G_{\mathrm{pd}}$ can
not be smaller than $G$ and thus $G=G_{\mathrm{pd}}$.

It remains to show that every point-determining graph is an F-core.
Consider two point-determining graphs $G\FullHomeq H$ that are not isomorphic. There are 
full homomorphisms $f:G\FullHom H$ and $g:H\FullHom G$. Because injective full
homomorphisms are embeddings, it follows that either $f$ or $g$ is not
injective.  Without loss of generality, assume that $f$ is not injective. Consider $u,v\in V_G$, $u\neq v$, such
that $f(u)=f(v)$. Because full homomorphisms preserve both edges and non-edges,
the preimage of any edge is a complete bipartite graph.
If we apply this fact on edges incident with $f(u)$, 
we derive that 
$N_G(u)=N_G(v)$, 
a contradiction with the assumption that $G$ is point-determining.
\end{proof}
Now we apply observations of Section~\ref{sec:embedding} together with the following proposition:

\begin{prop}
\label{prop:thincmp}
For F-cores $G$ and $H$ we have $G\FullHom H$ if and only if $G\EmbedHom H$.
\end{prop}
\begin{proof}
An embedding is a special case of a full homomorphism. In the opposite direction
consider a full homomorphism $f:G \FullHom H$. By the same
argument as in the proof of Proposition \ref{prop:F-core} we get that $f$ is injective,
as otherwise $G$ would not be point-determining.
\end{proof}

In other words, the orders $\FullLeq$, $\MonoLeq$ and $\EmbedLeq$ coincide on F-cores.
As an immediate consequence we get that the partial orders $(\Graphs,\FullLeq)$ and $(\DiGraphs,\FullLeq)$ are past-finite. 
Clearly, (directed) cycles are F-cores, hence the partial order $(\Cycles,\FullLeq)$ and $(\DiCycles,\FullLeq)$ are also past-finite-universal.

The following proposition was independently proved by Hell and Hern\'andez-Cruz~\cite[Theorem 1]{Hell2013}. We include a simple self-contained proof:

\begin{prop}[\cite{Hell2013}]
\label{prop:Fcoregap}
Every F-core $G$ with at least 2 vertices contains an $F$-core with $|V_G|-1$ vertices as an induced subgraph.
\end{prop}

\begin{proof}
It follows from Proposition~\ref{prop:thincmp} that the ordering $\FullLeq$ on $F$-cores
respects ordering by size: $H \FullLeq G$ implies $|V_H|<|V_G|$.

Hence, if there is a vertex $v$ of $G$ such that the graph $G\setminus v$ is point-determining, 
it is the desired F-core.

Assume to the contrary that for every vertex $v\in V_G$ there are at least two vertices
such that their neighborhoods are identical in $G\setminus v$. 

We construct an auxiliary graph $L$ on the vertices of $G$, where $u$ and $u'$ are adjacent if there 
is a vertex $v$ such that $N_{G\setminus v}(u)=N_{G\setminus v}(u')$. As for each edge of $L$ exists a unique $v$,
the graph $L$ has at least $|V_G|$ edges. Hence it contains a cycle, say $u_1,\dots,u_k$, $k\ge 3$.

Let $v_i$ be the vertex of $G$ whose removal makes neighborhoods of $u_i$ and $u_{i+1}$ identical.
By the uniqueness of vertices $v_i$, the neighborhoods $N_G(u_1)$ and $N_G(u_k)$ differ in vertices $v_1,\dots,v_{k-1}$.
On the other hand, as $u_1$ and $u_k$ are adjacent in $L$, their neighborhoods should differ in only one vertex, a contradiction.
\end{proof}

Observe, that from the Propositions~\ref{prop:thincmp} and \ref{prop:Fcoregap} immediately follows that 
there is no F-core strictly between F-cores $G$ and $G\setminus v$, hence such a pair forms an F-gap.
By analogous ideas we show that there are no other gaps in the full homomorphism order.

\begin{thm}[Gaps]
\label{thm:fullgap}
If $G$ and $H$ are F-cores and $(G,H)$ is an F-gap, then $G$ can be obtained from $H$
by removal of only one vertex. Consequently, the gaps in the full homomorphism order coincide
with the gaps in embedding order.
\end{thm}

Given a graph $G$ and a vertex $v\in V_G$ we denote by $G\setminus v$ the graph created from $G$ by removing vertex $v$.
We say that a vertex $v$ \emph{determines} a pair of vertices $u$ and $u'$, if  $N_{G\setminus v}(u)=N_{G\setminus v}(u')$.
This relation will play a key role in our analysis.  
We make use of the following Lemma:

\begin{lem}
\label{lem:determining}
Given a graph $G$ and a subset $A$ of the set of vertices of $G$ denote by $L$ a graph on the vertices of $G$, where $u$ and $u'$ are adjacent if and only if there 
is $v\in A$  that determines $u$ and $u'$. Let $S$ be any spanning tree of $L$.
Denote by $B\subseteq A$ the set of vertices that determine some pair of vertices connected by an
edge of $L$ and by $C\subseteq B$ set of vertices that determine some pair of vertices connected by 
an edge of $S$. Then $B=C$.
\end{lem}

\begin{proof}
For every pair of vertices there is at most one vertex determining them, hence $C\subseteq B\subseteq A$.

Assume to the contrary that there is a vertex $v\in B\setminus C$ and thus every
pair determined by $v$ is an edge of $L$ but not an edge of $S$.  Denote by $\{u,u'\}$  some such edge of $L$ determined by $v\in B$.
Adding this edge to $S$ closes a cycle. Denote by $u=v_1,
v_2,\ldots v_n=u'$ the vertices of $G$, such that every consecutive pair is an
edge of $S$.  Without loss of generality, we may assume that $v\in N_G(v_1)$ and $v\notin
N_G(v_n)$. Because $v\in N_G(v_i)$ implies $v\in N_G(v_{i+1})$ unless $v$
determines pair $\{v_i,v_{i+1}\}$ we also know that there is $1\leq i<n$,
such that $v$ determines $v_i$ and $v_{i+1}$. This is in contrary with the fact that
${v_i, v_{i+1}}$ forms an edge of $S$.
\end{proof}

\begin{proof}[Proof of Theorem~\ref{thm:fullgap}]
Assume to the contrary that there are F-cores $G$ and $H$ such that $(G,H)$ is
an F-gap, but $G$ differs from $H$ by more than one vertex.  By induction we construct two infinite sequences of vertices of $H$ denoted by 
$u_0,u_1,\ldots$ and $v_0,v_1,\ldots$ along with two infinite sequences of
induced subgraphs of $H$ denoted by $G_0,G_1,\ldots$ and $G'_0,G'_1,\ldots$ such that for every $i\geq 0$ it holds that:
\begin{enumerate}
\setlength\itemsep{0em}
\item $G_i$ and $G'_i$ are isomorphic to $G$,
\item $G_i$ does not contain $u_i$ and $v_i$,
\item $G'_i$ does not contain $u_i$ and $v_{i+1}$
\item $u_i$ and $u_{i+1}$ is determined by $v_i$, and,
\item $v_i$ and $v_{i+1}$ is determined by $u_i$.
\end{enumerate}

Put $G_0=G$ and $A=V_H\setminus V_G$.
Consider the spanning tree $S$ given by Lemma~\ref{lem:determining}.
 Because no vertex of $A$ can be removed to obtain an induced
point-determining subgraph, it follows that every vertex must have a
corresponding edge in $S$. Consequently the number of edges of $S$ is at least
$|A|$. Because $G$ itself is point-determining, it follows that every edge of
$S$ must contain at least one vertex of $A$.  These two conditions yield
a pair of vertices $v_0\in A=V_H\setminus V_G$ and $v_1\in V_G$ connected by an
edge in $S$ and also a vertex $u_0\in A$, which determines them. 
We have obtained $G_0, u_0, v_0, v_1$ with the desired properties.
This finishes the initial step of the induction.

Assume that at the induction step we have constructed $G_i, u_i, v_i, v_{i+1}$.
We show the construction of $G'_i$ and $u_{i+1}$. We distinguish two cases.
If $v_{i+1}\notin V_{G_i}$ we put $G'_i=G_i$.  If $v_{i+1}\in V_{G_i}$ we let $G'_i$ to be
the graph induced by $H$ on $(V_{G_i}\setminus\{v_{i+1}\})\cup \{v_i\}$.
Since the neighbourhood of $v_i$ and $v_{i+1}$ differs only by a vertex $u_i\notin G_i$, which
determines them, we know that $G'_i$ is isomorphic to $G_i$ (and thus also to $G$).
Moreover, $u_i$ is not a vertex of $G'_i$ (because $u_i\notin V_{G_i}$ cannot determine itself and thus $u_i\neq v_i$). If $H$ was point-determining
after removal of $v_{i+1}$, we would obtain a contradiction similarly as before. Hence we may assume that $v_{i+1}$
determines at least one pair of vertices.  Since the neighbourhoods of $v_{i+1}$ and of $v_i$ differ 
only by $u_i$,
we know that one vertex of this pair is $u_i$. Denote by $u_{i+1}$ the other vertex. 

We proceed similarly, when $G'_i, u_i, u_{i+1}$ and $v_{i+1}$ are given.
If $u_{i+1}\notin V_{G'_i}$, we put $G_{i+1}=G'_i$. If $u_{i+1}\in V_{G'_i}$, we let $G_{i+1}$ to be
the graph induced by $H$ on $(V_{G'_i}\setminus\{u_{i+1}\})\cup \{u_i\}$.
Again, $G_{i+1}$ is isomorphic to $G$ and does not contain $u_{i+1}$ nor $v_{i+1}$. Denote by $v_{i+2}$ 
the vertex determined by $u_{i+1}$ from $v_{i+1}$ (which again must exist by our assumptions). 
We have obtained $G_{i+1}, u_{i+1}, v_{i+1}, v_{i+2}$ with the desired properties.
This finishes the inductive step of our construction.

Because $H$ is finite, we know that both sequences $u_0,u_1,\ldots$  and
$v_0,v_1,\ldots$ contain repeated vertices.
Without loss of generality we may assume that the repeated vertex with the lowest index $j$ appears in the first sequence. We thus have $u_j=u_i$ for some $i<j$. By the minimality of $j$ we may assume
that $v_i,v_{i+1},\ldots v_{j-1}$ are all unique. Assume that $v_i$ is
in the neighbourhood of $u_i$. Then $v_i$ is not in the neighbourhood of $u_{i+1}$ (since it determines this pair) and consequently also $u_{i+1},u_{i+2},\ldots,u_{j}$.
We get a contradiction with $u_j=u_i$. If $v_i$ is not in the neighbourhood of $u_i$, we proceed analogously.
\end{proof}

Now we are ready to give a very simple proof of the existence of generalized dualities.
\begin{thm}[Dualities]
For every finite set of directed graphs $\mathcal{D}$ there is a finite set of graphs
$\mathcal{F}$ such that $(\mathcal{F},\mathcal{D})$ is a generalized finite F-duality pair.
Moreover, if $n=\max_{D\in\mathcal{D}} |V_D|$, then $\mathcal{F}$ may be chosen to consist of F-cores with at most $n+1$ vertices.
\end{thm}
\begin{proof}
Replace every member of $\mathcal{D}$ by its F-core; this does not increase $n$ and does not change the down-set generated by $\mathcal{D}$.  Let
\[
Y=\bigcup_{d\in\mathcal{D}}\downarrow d
\]
in the full homomorphism order on F-cores.  By Corollary~\ref{cor:pastfiniteduals}, it is enough to show that $P\setminus Y$ has finitely many minimal elements.

The empty graph is the unique minimum of the full homomorphism order.  Hence a minimal element of $P\setminus Y$ must be the upper endpoint $H$ of an F-gap $(G,H)$ with $G\in Y$.  For such a $G$ there is $D\in\mathcal{D}$ with $G\FullLeq D$.  By Proposition~\ref{prop:thincmp}, this means that $G$ embeds into the F-core $D$, and so $|V_G|\leq n$.  By Theorem~\ref{thm:fullgap}, the gap $(G,H)$ adds only one vertex; hence $|V_H|\leq n+1$.  There are only finitely many F-cores of size at most $n+1$, and they form a valid set $\mathcal{F}$ of left-hand obstructions.
\end{proof}

Essentially the same duality result (for undirected graphs)
was proved by Ball, Ne\v{s}et\v{r}il and Pultr~\cite{Ball2010} (stated in terms
of generalized finite dualities) and by Feder and Hell~\cite{Feder2008} (stated
as a bound on the obstruction to full $H$-colorability). The second was
extended to directed graphs by Hell and Hern\'andez-Cruz~\cite{Hell2013}
independently of our proof (Ball et al.~\cite{Ball2010} state their results 
in terms of relational structures so their approach covers directed graphs, too). 
A simple argument for the bound also appears in~\cite{Xie2006}. 
It is interesting to observe that all proofs are different.
Both \cite{Ball2010} and \cite{Xie2006} use Ramsey-type arguments. Our argument
is based on embeddings of point-determining graphs and is analogous to~\cite{Feder2008}, 
but we use a simpler argument in Proposition~\ref{prop:Fcoregap}
than the one using a characterization of nuclei in point determining graphs~\cite{Feder2008,Hell2013}.

\section{Surjective homomorphisms}
\label{sec:surjhomo}

Vertex and edge surjective homomorphisms are natural constrained variants.  Although less understood than ordinary homomorphisms, their complexity theory is active; see the survey~\cite{Bodirsky2012} and recent work on reflexive digraph targets, counting compactions and algebraic Sur-CSP methods~\cite{Larose2019SurjectiveReflexive,Focke2019CountingSurjective,Chen2024SurjectiveCSP}.

From our point of view, surjective homomorphisms are similarly easy as
embeddings and monomorphisms.  We consider three cases---vertex surjective
homomorphisms, edge-surjective homomorphisms, and surjective homomorphisms (that are both edge and vertex surjective).
\begin{prop}[Cores]
\label{prop:shomo-core}
Every finite graph is S-core and VS-core.

ES-core of graph $G$ is created from $G$ by removing all isolated vertices.
\end{prop}
\begin{proof}
Consider graph $G$ and its S-core or VS-core $H$.  By definition, there are
surjective homomorphisms $f:G\VSurHom H$ and $g:H\VSurHom G$.  By
surjectivity of $f$ we have $|V_G|\geq |V_H|$. From the surjectivity of $g$ we have
$|V_G|\leq |V_H|$ and thus $|V_G|=|V_H|$.  It follows that both $f$ and $g$ are
monomorphisms. Consequently $|E_G|\geq |E_H|\geq |E_G|$ and thus $f$ and $g$
are also isomorphisms.

On graphs without isolated vertices edge surjective homomorphisms are
surjective homomorphisms.  From the fact that removing isolated vertices does not affect
any edges we know that if $G'$ is created from $G$ by removing (some) isolated
vertices, then $G\SurHomeq G'$.
\end{proof}

It follows that $(\DiGraphs,\ESurLeq)$ is actually a suborder of
$(\DiGraphs,\SurLeq)$ that is induced on $(\DiGraphs,\SurLeq)$ by the class of
all ES-cores (that is graphs without isolated vertices).  

All the basic properties of the surjective homomorphism orders are almost immediate:

\begin{prop}[Future-finiteness]
\label{prop:shomo-univ}
The partial orders $(\DiGraphs,\SurLeq)$, $(\DiGraphs,\ESurLeq)$ and $(\DiGraphs,\VSurLeq)$ are future-finite.
\end{prop}
\begin{proof}
To see that $(\DiGraphs,\VSurLeq)$ is future-finite, fix graph $G$ and consider its up-set.
Since the homomorphism is vertex surjective, any graph $H$
in the up-set of $G$ has no more vertices than the graph $G$.

$(\DiGraphs,\ESurLeq)$ is a suborder of $(\DiGraphs,\SurLeq)$, i.e. a non-induced suborder of $(\DiGraphs,\VSurLeq)$.
\end{proof}
\begin{prop}[Universality]
The following partial orders $(\DiCycle,\SurLeq)$, $(\DiCycle,\ESurLeq)$ and $(\DiCycle,\VSurLeq)$ are all future-finite-universal.
\end{prop}
\begin{proof}
Observe that any homomorphism between two oriented cycles is also a surjective homomorphism.
The universality follows from Theorem~\ref{thm:cycles}.
\end{proof}
\begin{prop}[Gaps]
\label{prop:surgaps}
$(G,H)$ is a VS-gap if $|G|=|H|$ and $G$ is created from $H$ by removing an edge or
if $|G|=|H|+1$, $G\VSurHom H$ and there is no way to add an edge to $G$ preserving this property.

$(G,H)$ is a S-gap if $|G|=|H|+1$, $|E_G|=|E_H|$, $G\SurHom H$.

ES-gaps corresponds to S-gaps for pairs of graphs with no isolated vertices.
\end{prop}
\begin{proof}
It is easy to see that in all cases $|G|$ is at most $|H|+1$ vertices.
Otherwise graph in between can be constructed by partly concatenating vertices
of $G$ as given by the surjective homomorphism $G\SurHom H$. The extra condition
given prevents existence of a graph in between $G$ and $H$ in this case.

For vertex surjective mapping, there are gaps between $(G,H)$, $|G|=|H|$: the
mapping must be monomorphism. If the graphs $G$ and $H$ differs by precisely one
edge, they represent a gap in the monomorphism order.
In the edge surjective case however the
mapping must be embedding and there are no gaps in the embedding order such that
$|G|=|H|$.
\end{proof}
\begin{prop}[Dualities]
\label{prop:surdual}
For every finite set of directed graphs $\mathcal{F}$ there are finite sets of directed graphs
$\mathcal{D}_S$, $\mathcal{D}_{VS}$ and $\mathcal{D}_{ES}$, such that $(\mathcal{F},\mathcal{D}_S)$ is a generalized finite S-duality pair,
$(\mathcal{F},\mathcal{D}_{VS})$ is a generalized finite VS-duality pair, and
$(\mathcal{F},\mathcal{D}_{ES})$ is a generalized finite ES-duality pair.
\end{prop}
\begin{proof}
Again the existence of dualities follows from the future-finite duality criterion, Proposition~\ref{prop:futurefiniteduals}, and from the characterization of gaps (Proposition~\ref{prop:surgaps}).
Observe that a single vertex with a loop on it and the empty graph are the only maximal elements of the vertex surjective homomorphism
order.  In edge surjective homomorphisms there are three maximal elements; single vertex, vertex with a loop and the empty graph.
For the S-order and, after deleting isolated vertices, for the ES-order, the relevant gaps differ by one vertex; the VS-order also has the edge-deletion gaps described in Proposition~\ref{prop:surgaps}.
\end{proof}

Note that loops have to be allowed to make finite dualities possible: 
for graphs without loops the set $\mathcal{D}$ would need to contain cliques of
arbitrary sizes, because all cliques are maximal elements in surjective homomorphism
order on (directed) graphs without loops. Thus there are no dualities in the case of
loopless graphs.

\section{Locally constrained homomorphisms}
\label{sec:locally-constrained}

Locally constrained homomorphisms have a topological origin: locally bijective homomorphisms, called graph covers, are a discrete variant of covering projections between topological spaces. These were studied already in the 1930s; see the monograph of Reidemeister~\cite{k:Reidemeister32}. 

Locally bijective homomorphisms were involved in constructions of highly transitive graphs~\cite{k:Biggs74} or classification of maps~\cite{n:MNS02}. In 1980 Angluin and later Leighton~\cite{n:Angluin80, Leighton1982} showed a construction of a common cover (see Corollary~\ref{cor:singleclass} (iv) below), which can be viewed as the first order-related result.

Locally injective homomorphisms (sometimes called partial covers) and locally surjective mappings (known as role assignments) are 
straightforward generalizations of locally bijective ones. All three are defined by a particular behavior on the neighborhoods of vertices and thus share many common properties. There are also a number of other non-trivial relations between the three kinds of homomorphisms. See \cite{Fiala2008} for a survey of this area; for more recent algorithmic developments see \cite{Bulteau2024Algorithmic,Dvorak2022ListLS}.  Graph covers themselves also continue to be studied in connection with covering and semi-covering problems~\cite{KratochvilNedela2025CoversSemicovers}.
The order of locally constrained homomorphisms was studied in \cite{Fiala2005}. We review and extend these results.

It is well known that the existence of a locally bijective homomorphism between two directed graphs
can be decided by comparing two undirected graphs obtained from the directed ones, and conversely~\cite{Welzl1982}. 
To our knowledge, locally injective and locally surjective homomorphisms have not been considered for
directed graphs; however, a construction that works for locally bijective homomorphisms may work as well.

We first focus on properties of locally constrained homomorphism orders on 
the class $\ConnGraph$, containing all finite connected undirected graphs. 
Throughout this section, unless explicitly stated otherwise, graphs in $\ConnGraph$ are finite, undirected and loopless, and $N_G(v)$ denotes the usual open neighborhood. 
Note that the partial orders induced on $\ConnGraph$ may maintain some properties that are not preserved on $\Graphs$. 
For example, one can observe that locally surjective homomorphism between two connected graphs is also surjective 
(globally, i.e. on the whole vertex set), while this might be violated when the target graph is disconnected.

The following results show the main correspondence between all three types of locally constrained homomorphisms.

\begin{thm}[Fiala, Maxov\'a \cite{Fiala2006}]
If a graph $G$ admits a locally injective homomorphism $f$ to a finite and connected graph $H$ as well as a locally surjective homomorphism $g$ to $H$, then all locally constrained homomorphisms between $G$ and $H$ are locally bijective.
\end{thm}
Or in the language of homomorphism orders we get:
\begin{thm}[Fiala, Paulusma, Telle \cite{Fiala2005}]
\label{thm:intersect}
Partial order $(\ConnGraph,\LocBiLeq)$ is the intersection of partial orders $(\ConnGraph,\LocSurLeq)$ and $(\ConnGraph,\LocInLeq)$.
\end{thm}

\subsection{Common properties}

Consider a locally bijective homomorphism $f:G\LocBiHom H$.  By definition, for every vertex $x$, the mapping $f$ induces an isomorphism from the neighborhood $N_G(x)$ of $x$, to the neighborhood $N_H(f(x))$ of $f(x)$. 
Consequently $x$ and $f(x)$ have the same degree. Moreover, the image $f(x')$ of a vertex $x'\in N_G(x)$ must be a vertex of the same degree in the neighborhood of $f(x)$. It follows that any locally bijective homomorphism preserves not only the degrees of vertices, but also the degrees of neighborhood vertices and so on. This property is captured by the notion of degree refinement matrices.

A partition of the vertex set of a graph $G$ into disjoint classes is called an {\em equitable partition} if the vertices in the same class have the same numbers of neighbors in all classes of the partition.

Any equitable partition is characterized by the associated {\em degree matrix} whose rows and columns are indexed by the blocks of the partition, and the entry in the $i$-th row and $j$-th column describes how many neighbors a vertex of the $i$-th block has in the $j$-th block.

Every finite graph $G$ admits a unique minimal equitable partition. In this case a canonical ordering can be imposed on the blocks, so the corresponding degree matrix, called the {\em degree refinement matrix}, $\drm(G)$, is also defined uniquely, for formal details see a survey~\cite{Fiala2008}.

In the sequel we utilize the following folklore necessary conditions~\cite{Fiala2008}:

\begin{prop}\label{prop:folklorecover}
If $G\LocBiHom H$ with $G,H \in \ConnGraph$, then
\begin{itemize}
	\item $\drm(G)=\drm(H)$, and
	\item $|V_G|$ is an integer multiple of $|V_H|$.
\end{itemize}
\end{prop}

We put $G\Matrixeq H$ if and only if $\drm(G)=\drm(H)$. The relation $\Matrixeq$ is an equivalence relation on the class of finite graphs, and by the proposition the relation $\LocBiHom$ is a sub-relation of $\Matrixeq$ on $\ConnGraph$.

Moreover, the equivalence classes $\K$ of $(\ConnGraph,\Matrixeq)$ could be categorized as follows:
If $\K$ contains a tree, then $\K$ is trivial (i.e. having only a single graph up to an isomorphism).
Otherwise $\K$ contains infinitely many nonisomorphic graphs.

The connections between $\Matrixeq$ and all three variants of locally constrained homomorphisms are captured by the following two results.

\begin{thm}[Fiala, Kratochv\'il \cite{Fiala2001}]
\label{thm:degeq1}
If a graph $G$ and a connected graph $H$ share the same degree matrix, then every locally injective homomorphism is locally bijective.
\end{thm}
\begin{thm}[Kristiansen, Telle \cite{Kristiansen2000}]
\label{thm:degeq2}
If a graph $G$ and a connected graph $H$ share the same degree matrix, then every locally surjective homomorphism is locally bijective.
\end{thm}

As a consequence we may strengthen Theorem~\ref{thm:intersect}, by observing that within each single equivalence class of 
$(\ConnGraph,\Matrixeq)$ all three locally constrained homomorphisms coincide.
It is therefore natural to explore the common properties of these three orders within a single nontrivial equivalence class.
Consider the following construction:

For a nontrivial $\K$ choose $G_1\in \K$ arbitrarily. As $G_1$ is not a tree, it contains a cycle. Let $(u,v)$ be an edge of this cycle. For any $k\in \mathbb{N}$ we construct $G_k$ from $k$ disjoint copies of $G_1$, where the edge $(u_i,v_i)$ in the $i$-th copy
is replaced by the edge $(u_i,v_{i+1})$ between the consecutive copies (the index at $v$ is taken modulo $k$).
We refer to $G_k$ as a cyclic $k$-fold cover of $G_1$.

The following claims are immediate:
As $G_i\setminus (u,v)$ was connected, all graphs $G_k$ are connected.
When $l$ is a multiple of $k$, then $G_l$ is a cyclic $\frac{l}{k}$-fold cover of $G_k$.
Otherwise, $G_l$ and $G_k$ are incomparable in $\LocBiLeq$ by Proposition~\ref{prop:folklorecover} (ii).
Hence the assignment $k\to G_k$ is an embedding of $(\mathbb{N},\overleftarrow|)$ in $(\K,\LocBiLeq)$.

We summarize the properties of the partial orders 
within a single equivalence class of $\Matrixeq$ as follows:

\begin{corollary}
\label{cor:singleclass}
Let $\K$ be a nontrivial equivalence class of $(\ConnGraph,\Matrixeq)$. Then
the partial orders $(\K,\LocBiLeq)$, $(\K,\LocSurLeq)$ and $(\K,\LocInLeq)$ are all equivalent 
and have the following properties:
\begin{enumerate}
 \item there are no minimal elements;
 \item all three partial orders are future-finite;
 \item all three partial orders are future-finite-universal;
 \item for every pair of graphs $G,H\in \K$ there exists a graph $C\in \K$ such that $C\LocBiLeq G$ and $C\LocBiLeq H$.
\end{enumerate}
\end{corollary}

\begin{proof}
The equivalence of partial orders $(\K,\LocBiLeq)$, $(\K,\LocSurLeq)$ and $(\K,\LocInLeq)$ follows from Theorems \ref{thm:degeq1} and \ref{thm:degeq2}.

Items (i-iii) are inherited from the order $(\mathbb{N},\overleftarrow|)$, for the latter two see Proposition~\ref{prop:setdiviuniv}.

For (iv), a non-trivial method of constructing graphs $C$ is shown in
\cite{Leighton1982}, an alternative proof in terms of Bass--Serre theory appears in
\cite{Bass1990} and \cite{Neumann2011}. 
\end{proof}

Observe that 2-regular undirected graphs are undirected cycles and thus $\Cycle$ is an equivalence
class of $(\ConnGraph,\Matrixeq)$ characterized by $\drm=(2)$.  In particular, we get:

\begin{prop}
\label{pro:lochomo-univ}
Partial orders $(\Cycle,\LocBiLeq)=(\Cycle,\LocSurLeq)=(\Cycle,\LocInLeq)$ are future-finite-universal.
\end{prop}

Indeed, the partial order $(\UnorPath,\LocSurLeq)$,
where $\UnorPath$ is the class of all paths, is future-finite-universal as $P_l$
allows a locally surjective homomorphism to $C_k$ if and only if $k$ divides $l$.

Finally note that $(\ConnGraph,\LocBiLeq)$ can be augmented by specific infinite trees so
that each equivalence class $\K$ will contain a unique minimal element of $\LocBiLeq$ ---
whenever $\K$ is nontrivial, then there exists graph $C_{\K}$ (known as the universal cover) such that for every $G\in \K$, $C_{\K}\LocBiLeq G$~\cite{Leighton1982}.

We now pass from connected graphs to possibly disconnected graphs and combine Proposition~\ref{pro:lochomo-univ} 
and Theorem~\ref{thm:univ} to a stronger universality result.

\begin{corollary}[Universality]
\label{cor:univgraphs}
Let $\K$ be a nontrivial equivalence class of $(\Graphs,\Matrixeq)$. Then
the partial order $(\K,\LocBiLeq)$ is universal, and hence also $(\K,\LocSurLeq)$ and $(\K,\LocInLeq)$.
\end{corollary}

In particular, partial orders $(\Cycles,\LocBiLeq)=(\Cycles,\LocSurLeq)=(\Cycles,\LocInLeq)$ are universal.

We are not aware of any reasonable characterization of gaps in locally bijective order and doubt that any exists 
due to the following construction.

Let $G$ be a $(r-1)$-regular graph on $t\cdot r$ vertices. Choose a prime $p$ greater than $t$ and augment
$G$ with $p-t$ copies of $K_r$. Join these graphs together similarly as the construction of cyclic $(p-t+1)$-fold cover
to obtain a graph $G'$. It is straightforward to argue that $G \LocBiHom K_r$ if and only if $G' \LocBiHom K_r$. 
By Proposition~\ref{prop:folklorecover} (ii) the latter is equivalent to saying that $(G',K_r)$ is a gap in $(\Graphs,\LocBiLeq)$.
As the decision problem whether $G \LocBiHom K_r$ for $r\ge 4$ is {\sf NP}-complete,
it is unlikely that such a characterization of gaps would be computationally efficient.

On the other hand, the orders considered have many gaps, as Proposition~\ref{prop:folklorecover} yields:

\begin{corollary}
\label{cor:locbigaps}
Let $G_1$ be a graph with a cycle. Then for every prime $p$, any
cyclic $p$-fold cover $G_p$ forms with $G_1$ an LB-gap, an LI-gap and also an LS-gap.  With our convention that $G\leq H$ means that there is a homomorphism $G\to H$, this gap is the pair $(G_p,G_1)$.
\end{corollary}

\subsection{Locally surjective and locally bijective homomorphisms}

For connected graphs $G$ and $H$, it is easy to observe that both locally bijective and locally surjective homomorphisms are also surjective homomorphisms (but not vice versa; surjective homomorphism may not be locally surjective homomorphism). 
In such a situation we get the following order properties:

By Proposition \ref{prop:shomo-core} we have:

\begin{prop}[Cores]
\label{prop:locsurbi-core}
Every finite connected graph is LS-core and LB-core.
\end{prop}

By Proposition \ref{prop:shomo-univ} we have:

\begin{prop}[Future-finiteness]
\label{pro:lochomo-fini}
Partial orders $(\ConnGraph,\LocBiLeq)$ and $(\ConnGraph,\LocSurLeq)$ are future-finite.
\end{prop}

\begin{lem}
\label{lem:locsurgaps}
For any tree $T$ there are infinitely many graphs $H$ such that $(H,T)$ is an LS-gap.
\end{lem}

\begin{proof}
Let us first consider the simple case of a path $P_n$ with $n$ edges. 
It is not difficult to see that every path $P_{pn}$ where $p$ is a prime number
forms a gap: the only graphs in between $P_n$ and $P_{pn}$ and locally surjective
homomorphism of two paths must map the initial vertex to the initial or terminal vertex 
and the terminal vertex to the initial or terminal vertex. Moreover, it cannot flip
in the middle of the path.

In general, each tree $T$ has at least two leaf vertices of degree one. 
Denote two leaves of the tree as the initial and the terminal vertex and connect $p$
copies of $T$ into a sequence to form a gap.
\end{proof}

\begin{prop}[Dualities]
\label{prop:locbidual}
There are no generalized finite duality pairs in $(\ConnGraph,\LocBiLeq)$
and in $(\ConnGraph,\LocSurLeq)$.
\end{prop}
\begin{proof}
We apply Proposition~\ref{prop:futurefiniteduals}.  

In the locally bijective case the analysis is easier.  There are no
homomorphisms between graphs $G$, $H$ such that $\drm(G)\neq \drm(H)$ and
because there are infinitely many degree refinement matrices and thus there
are infinitely many maximal elements in $(\ConnGraph,\LocBiLeq)$.

In the case of locally surjective order, there are infinitely many gaps below graphs with a cycle 
(Corollary~\ref{cor:locbigaps}) as well as below acyclic graphs (Lemma~\ref{lem:locsurgaps}).
Hence there are no duality pairs.
\end{proof}

\subsection{Locally injective homomorphisms}

We now show that the locally injective homomorphism order has a number of properties 
that are in sharp contrast with the other two locally constrained orders.

To characterize cores we apply the following classical result (that is a special case of Theorem~\ref{thm:degeq1} but proved three decades earlier).

\begin{thm}[Ne\v set\v ril, \cite{Nesetril1971}]
\label{thm:locin-auto}
Let $G$ be a finite connected graph. Every locally injective homomorphism $f:G\LocInHom G$ is an automorphism of $G$.
\end{thm}

\begin{corollary}[Cores]
\label{cor:locin-core}
Every finite connected graph is an LI-core.
\end{corollary}

\begin{proof}
The mappings $f:G\LocInHom H$ and $f':H\LocInHom G$ can be composed into locally injective homomorphisms $f'\circ f:G\LocInHom G$ and $f\circ f':H\LocInHom H$.  By Theorem~\ref{thm:locin-auto}, $f'\circ f$ is an automorphism of $G$, hence surjective; therefore $f'$ is surjective and $|V_H|\geq |V_G|$.  Similarly, $f\circ f'$ is an automorphism of $H$, hence surjective; therefore $f$ is surjective and $|V_G|\geq |V_H|$. Consequently $|V_G|=|V_H|$ and both $f$ and $f'$ are isomorphisms.
\end{proof}

For the first (and also for the last) time we can prove a variant of Welzl's density theorem~\cite{Welzl1982}. 
This shows that the locally injective homomorphism order on undirected connected graphs is not locally finite.

\begin{thm}[Density]
\label{thm:density}
Let $G$ and $H$ be connected graphs such that $\drm(G)\neq \drm(H)$, $G\LocInL H$ and $H$ has minimum degree~2. 
\begin{enumerate}
\item[(a)] Then a connected graph $F$ exists, such that $G\LocInL F\LocInL H$, $\drm(G) \neq \drm(F) \neq \drm(H)$. 
\item[(b)] If $G$ has minimum degree~2 and $H$ is not isomorphic to a cycle, then $F$ can also be chosen of 
minimum degree~2 and nonisomorphic to a cycle.
\end{enumerate}
\end{thm}

\begin{proof}
Fix a locally injective homomorphism $g:G\to H$.  Since $\drm(G)\neq \drm(H)$, the mapping $g$ is not locally bijective.  
Choose a vertex $u$ such that the induced map $N_G(u)\to N_H(g(u))$ is injective but not surjective, and choose a vertex $w\in N_H(g(u))\setminus g(N_G(u))$.

First construct $F$ as a copy of the graph $G$ extended by a single vertex $u'$ and a single edge $(u,u')$. 
We claim that $F$ satisfies the conditions given on $F$ by (a). 
Since $G$ is a subgraph of $F$, the embedding yields a locally injective homomorphism $G\LocInHom F$.
Also $F\LocInHom H$: extend $g$ to $f:F\to H$ by putting $f(u')=w$.  The choice of $w$ makes the map injective on $N_F(u)$, and local injectivity at the new vertex $u'$ is immediate. 
Moreover, the matrix $\drm(F)$ differs from $\drm(H)$ because $F$ has a vertex of degree 1 while $H$ has no vertex of degree 1.  
For the last claim suppose by a contradiction that $\drm(F) = \drm(G)$. 
From Theorem \ref{thm:degeq1} the embedding obtained from the strict inclusion would be a locally bijective homomorphism
from $G$ to $F$, but the required properties are violated on $N(u)$. 

Now we extend our construction to meet requirements of (b). In the graph $H'=H\setminus (g(u),w)$
we consider the component containing $w$. This component contains a cycle $C$ 
as otherwise the component would contain also a vertex of degree one and then
either $H$ would be isomorphic to a cycle (if the only other vertex of degree one would be $g(u)$) 
or the minimum degree of $H$ would be one (otherwise); both cases are in contradiction with our assumptions.
Let $c$ be the length of $C$ and $l$ be the distance of $w$ to the nearest vertex of $C$ in $H'$.

Choose an integer $c'$ to be a multiple of $c$ that is at least twice greater than the orders of matrices $\drm(G)$ and $\drm(H)$.
We now append to the so far constructed graph $F$ an extra copy of a cycle of length $c'$.  One of its vertices, say $v$, is joined to $u'$ by a path of length $l$; when $l=0$ this means that $v$ is identified with $u'$.
This subgraph attached to $u$ is an {\em $(l+1,c')$-lasso}, see Figure~\ref{fig:lasso}.

\begin{figure}
\centerline{\includegraphics{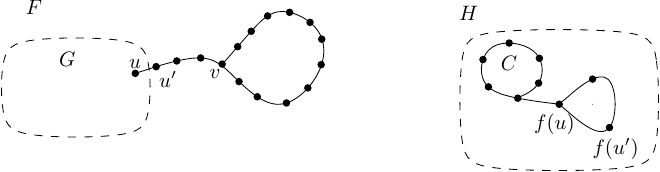}}
\caption{Left: the graph $F$ formed from $G$ by attaching a $(4,12)$-lasso to the vertex $u$.
Right: the choice of $C$ in $H$; note that $g(u)$ may be on the shortest path between 
$w$ and $C$ in $H'$ as indicated.}
\label{fig:lasso}
\end{figure}

As before, $G\LocInL F$ because $G\subseteq F$. 
The map $f$ can further be extended by mapping $v$ to the nearest vertex of $C$ to $w$ in $H'$, wrapping the cycle of the lasso $c'/c$ times around $C$, and mapping the path from $u'$ to $v$ onto a shortest path from $w$ to $C$ in $H'$.  When $l=0$, this path is empty and $u'=v$ is mapped to $w\in V(C)$.
Such an extension satisfies all local constraints.  At $u$ this follows from the choice of $w\notin g(N_G(u))$.  At $v$, the two cycle neighbors are mapped to the two neighbors on $C$, while the third neighbor is mapped either to the preceding vertex of the shortest path from $w$ to $C$ when $l>0$, or to $g(u)$ when $l=0$; in both cases the three images are distinct.  Thus $F\LocInL H$.

Since $v$ is the only vertex of degree 3 on the lasso, then 
two vertices of the lasso that have different distance to $v$ 
must belong to different blocks of any equitable partition.
Therefore the order of $\drm(F)$ is at least $\frac{c'}2$.
Due to the choice of $c'$, this value is greater than the orders of $\drm(G)$ and $\drm(H)$.
Consequently, $\drm(F)$ is identical neither with $\drm(G)$ nor with $\drm(H)$.
\end{proof}

\begin{corollary}[Infinite-chain density]
\label{cor:locin-infinite-chain}
Under the assumptions of Theorem~\ref{thm:density}(a), the interval
\[
\{F\mid G\LocInL F\LocInL H\}
\]
contains an infinite chain.  Under the additional assumptions of Theorem~\ref{thm:density}(b), the chain can be chosen among connected graphs of minimum degree at least two that are not cycles.
\end{corollary}

\begin{proof}
Apply Theorem~\ref{thm:density} to $(G,H)$ to obtain $G\LocInL F_1\LocInL H$ with $\drm(F_1)\neq\drm(H)$.  Since $H$ is unchanged and has minimum degree at least two, the theorem applies again to $(F_1,H)$, yielding $F_1\LocInL F_2\LocInL H$.  Iterating gives a strictly increasing chain
\[
G\LocInL F_1\LocInL F_2\LocInL F_3\LocInL\cdots\LocInL H.
\]
The additional hypotheses in part (b) are preserved at each step by the construction in Theorem~\ref{thm:density}(b).
\end{proof}

The constraints imposed on the graph $H$ may seem artificial. Nevertheless, they lead to full characterizations of graphs $H$ such that there is $G$, $\drm(G)\neq \drm(H)$ and $(G,H)$ is a gap.

It remains to explore the cases where $H$ has neither vertex of degree 1, nor a cycle with vertex of degree greater than 2.  We know that $H$ is not a tree and thus it contains a cycle. Because $H$ is connected it follows that all such graphs $H$ are cycles.
Graphs with locally injective homomorphism to a cycle are either cycles or paths.  When $G$ is a path of length $l$, the graphs $F$ can be chosen as a path of length $l+1$. It is easy to observe that $G\LocInHom F\LocInHom G$ and the degree matrices of $G,H$ and $F$ are mutually disjoint.
If $G$ is a cycle, we have $\drm(G)=\drm(H)$ and thus we cannot hope for a density result in general. As an example, consider $G$ to be a cycle of length $4$ and $H$ to be a cycle of length $8$.

Combining all results together we get:

\begin{thm}[Gaps]
Let $H$ be a connected graph.
\begin{itemize}
\item[(a)] There exists a connected graph $G$, such that $\drm(G)=\drm(H)$ and $(G,H)$ is a gap in $(\ConnGraph,\LocInLeq)$ if and only if $H$ contains a cycle.  In this case $G$ may be chosen as a cyclic $2$-fold cover of $H$.
\item[(b)] There exists a connected graph $G$, such that $\drm(G)\neq \drm(H)$ and $(G,H)$ is a gap if and only if $H$ has at least one vertex of degree 1.  In this case one may take $G=H\setminus v$ for any leaf $v$ of $H$.
\end{itemize}
\end{thm}

\begin{proof}
The equivalence class of an acyclic $H$ is trivial, which yields the forward implication of $(a)$.

In the opposite direction, we may choose $G$ as a cyclic 2-fold cover of $H$. 
Theorem \ref{thm:degeq1} implies any graph $F$ such that $G\LocInHom F\LocInHom H$ satisfies $\drm(F)=\drm(G)=\drm(H)$, 
and indeed $G\LocBiHom F\LocBiHom H$. The second claim of Proposition~\ref{prop:folklorecover} and the fact $|V_G|=2|V_H|$
yields that $F$ is isomorphic either to $G$ or $H$.

The backward implication of $(b)$ is a direct consequence of Theorem \ref{thm:density} (a) since 
no graph $H$ without vertices of degree one allows a gap $(G,H)$ where $\drm(G)\neq \drm(H)$.

In the other direction, let $v$ be the vertex of degree one in $H$. We show that $G=H\setminus v$ has the desired properties. 
Assume to the contrary that an $F$ exists such that $G\LocInL F\LocInL H$.

When $H$ is acyclic, then $F$ is also acyclic and both locally injective homomorphisms are embeddings.  
Consequently, $G$ is a proper subgraph of $F$, which is a proper subgraph of $H$ --- a contradiction with $|V_G|+1=|V_H|$.

Now assume that $H$ contains a cycle. Let $H'$ and $G'$, resp., be the largest induced subgraphs of $H$ and $G$, resp., containing no vertices of degree one. By the construction of $G$ we get that $G'$ and $H'$ are isomorphic. Moreover, the components induced by 
$E_H\setminus E_{H'}$ are trees as $H'$ can be obtained by iterative removal of vertices of degree one.
Consider any such component $T$ and any locally injective homomorphism $f:G\LocInHom H$.
We claim that no edge $e$ of $G'$ can be mapped by $f$ to an edge of $T$ --- as $e$ belongs to a closed trail in $G'$ and 
the image of a trail is also a trail, it cannot be mapped inside $T$, as $T$ is a tree.
Consequently $f$ maps $G'$ onto $H'$ and by Theorem~\ref{thm:locin-auto}, it is an isomorphism of these subgraphs.
Moreover, $f$ maps the components induced by $E_G\setminus E_{G'}$ injectively into components of $E_H\setminus E_{H'}$
so any $f: G\LocInHom H$ is an embedding.

Now consider any graph $F: G\LocInL F\LocInL H$ and homomorphisms $g: G\LocInHom F$, $h: F\LocInHom H$.
As $h\circ g: G\LocInHom H$ we get from the above argument that $h\circ g$ is an embedding, hence also $g$ is an embedding
and $h$ is an isomorphism between $g(G')$ and $H'$.
Observe that components of $E_F\setminus E(g(G'))$ are mapped by $h$ onto trees in $E_H\setminus E_{H'}$. 
By arguments analogous to those above we obtain that $h$ is also an embedding, which implies that $F$ is isomorphic either to $G$ or to $H$.
\end{proof}

Further developing these ideas, we now show --- as the main technical contribution of this section --- the universality of the
order given by locally injective homomorphisms between \emph{connected} graphs 
(compare with the simplicity of Corollary~\ref{cor:univgraphs} for the disconnected case).

\begin{thm}[Universality]
\label{thm:lochomouniv}
$(\ConnGraph,\LocInLeq)$ is a universal partial order.
\end{thm}

Consider any partial order $(P,\leq_P)$. We construct an on-line
embedding from $(P,\leq_P)$ to $(\ConnGraph,\LocInLeq)$ by representing the
elements by cycles similarly as in Theorem~\ref{thm:cycles}. 
The main complication is that the cycles representing an element of $(P,\leq_P)$ have to be connected into a single graph without creating unwanted locally injective homomorphisms.

\begin{proof}
Without loss of generality we assume that $P=\{p_1,p_2,\ldots\}$ consists of prime numbers strictly greater than 3.  We use the usual order of the primes only to run the induction; the given partial order is denoted by $\leq_P$.
For every $n\in P$ put
\[
p(n)=\prod\{q\mid q\in P,\ q\leq n \hbox{ and } q\geq_P n\},
\qquad
 a(n)=p(n)2^{n-1}.
\]
For $n>n'$ in the usual order,
\[
 n\leq_P n' \quad\Longleftrightarrow\quad p(n')\mid p(n).
\]
Indeed, if $p(n')\mid p(n)$, then the prime $n'$ occurs in $p(n)$, so $n'\geq_P n$; the converse follows by transitivity of $\leq_P$.  The additional factor in $a(n)$ is chosen so that, whenever $n>n'$ and $p(n')\mid p(n)$,
\[
 p(n')2^{n'}\mid a(n).                                      
\]
This divisibility is the correction that makes the gluing argument below work.

For $n\in P$, let $H(n)$ be the graph formed from two cycles of lengths
\[
 L_n=p(n)2^n \quad \hbox{and} \quad R_n=3\cdot2^n,
\]
joined by a path of length 2.  We denote the vertices of the first cycle by
$l_0,l_1,\ldots,l_{L_n-1}$ and call it the left cycle; the vertices of the other cycle are
$r_0,r_1,\ldots,r_{R_n-1}$ and form the right cycle.  The joining path connects
$l_{a(n)}$ and $r_0$.  Finally, we add pendant vertices to all remaining cycle vertices so that every vertex on either cycle has degree three.

\begin{lem}\label{lem:sunlets}
For any $n,n'\in P$ with $n>n'$, a locally injective homomorphism
$H(n)\LocInHom H(n')$ exists if and only if $n\leq_P n'$.  Moreover, whenever such a homomorphism exists, every locally injective homomorphism $H(n)\LocInHom H(n')$ maps $l_0^n$ to $l_{a(n')}^{n'}$ and $r_0^n$ to $r_0^{n'}$.
\end{lem}

\begin{proof}
The vertices of degree three in $H(n)$ lie exactly on the two cycles, and the remaining vertices have degree one or two.  Hence any locally injective homomorphism maps each cycle of $H(n)$ onto a cycle of $H(n')$.  Its restriction to a cycle is a covering of cycles; in particular, a cycle of length $L$ can be mapped locally injectively to a cycle of length $L'$ only when $L'$ divides $L$.

The left and right cycles cannot be interchanged.  The left cycle lengths are not divisible by 3, while every right cycle length is divisible by 3; hence a left cycle of $H(n)$ cannot cover the right cycle of $H(n')$.  Conversely, the left cycle of $H(n')$ has the prime divisor $n'>3$, whereas the right cycle of $H(n)$ has only the prime divisors 2 and 3; hence the right cycle of $H(n)$ cannot cover the left cycle of $H(n')$.  Thus every locally injective homomorphism $f:H(n)\LocInHom H(n')$ preserves the left-right position of the two cycles.

Let the superscripts indicate the ambient graph.  The unique path of length 2 between the two cycles in $H(n)$ must be mapped to the unique such path in $H(n')$.  Consequently
\[
 f(r_0^n)=r_0^{n'} \quad\hbox{and}\quad f(l_{a(n)}^n)=l_{a(n')}^{n'} .
\]
The restriction of $f$ to the left cycle therefore exists only if
\[
 L_{n'}=p(n')2^{n'} \mid p(n)2^n=L_n,
\]
which, because $n>n'$, is equivalent to $p(n')\mid p(n)$ and hence to $n\leq_P n'$.  Moreover, under this condition $a(n)=p(n)2^{n-1}$ is divisible by $L_{n'}$.  Therefore, regardless of the orientation of the covering of the left cycle,
\[
 f(l_0^n)=f(l_{a(n)}^n)=l_{a(n')}^{n'} .
\]

Conversely, assume $n\leq_P n'$, equivalently $p(n')\mid p(n)$.  Define $f$ on the cycle vertices by
\[
 f(l_i^n)=l_{(a(n')+i) \bmod L_{n'}}^{n'},\qquad
 f(r_i^n)=r_{i \bmod R_{n'}}^{n'} .
\]
Since $L_{n'}\mid a(n)$ and $R_{n'}\mid R_n$, this sends $l_{a(n)}^n$ to $l_{a(n')}^{n'}$ and $r_0^n$ to $r_0^{n'}$.  Map the middle vertex of the joining path to the middle vertex of the joining path of $H(n')$.  Finally, each pendant vertex adjacent to a cycle vertex $x$ is mapped to the unique neighbor of $f(x)$ not used by the two cycle neighbors of $x$.  This gives a homomorphism, and at every source vertex the images of its neighbors are distinct; hence the homomorphism is locally injective.
\end{proof}

\begin{figure}
\centerline{\includegraphics{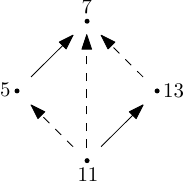}}
\caption{Partial order $(P,\leq_P)$.}
\label{fig:poset2}
\end{figure}
\begin{figure}
\centerline{\includegraphics{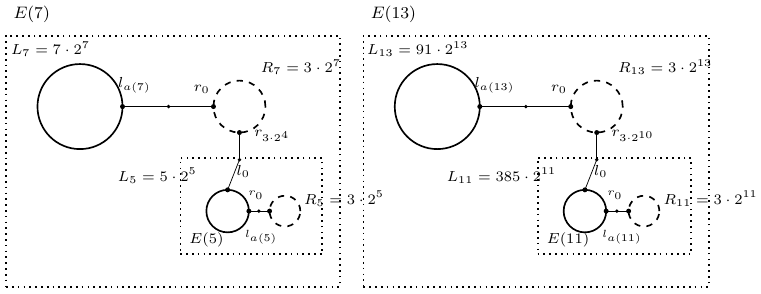}}
\caption{Representation of partial order $(P,\leq_P)$. Solid circles denote the left cycles and dashed circles the right cycles.  The marked joining path in $H(n)$ starts at $l_{a(n)}$ and ends at $r_0$; vertical attachments to earlier graphs use the vertices $r_{3\cdot 2^{k-1}}$.}
\label{fig:representation}
\end{figure}

We now complete the construction by encoding the forward order.
For any $n\in P$ we inductively construct a connected graph $E(n)$ from the disjoint union of $H(n)$ and all graphs $E(k)$ with $k<n$ and $k\leq_P n$.  The copies are connected as follows.  The notation $l_0^k$ always refers to the distinguished vertex $l_0$ in the top copy of $H(k)$ inside $E(k)$.  For each $k<n$ with $k\leq_P n$, identify the pendant neighbor of $l_0^k$ in $E(k)$ with the pendant neighbor of $r_{3\cdot2^{k-1}}^n$ in $H(n)$.

Graphs $E(p)$, $p\in P$, for the sample partial order $(P,\leq_P)$ shown in Figure~\ref{fig:poset2} are schematically depicted in Figure~\ref{fig:representation}.

The resulting graphs are cacti.  Their cycles are of two recognizable types: right cycles have length divisible by $3$, while left cycles have length divisible by a prime greater than $3$.  The bridge argument used in Lemma~\ref{lem:sunlets} shows that every locally injective homomorphism between two such graphs preserves this left-right distinction and maps cycles to cycles of the same type.

We prove by induction on the usual order of the primes that the assignment $n\mapsto E(n)$ embeds $(P,\leq_P)$ into $(\ConnGraph,\LocInLeq)$.  In the inductive construction we use the following strengthened form: whenever $x>y$ and $x\leq_P y$, the homomorphism $E(x)\LocInHom E(y)$ can be chosen so that its restriction to the top copy of $H(x)$ is the explicit map from Lemma~\ref{lem:sunlets}.

Let $n'<n$.

\begin{itemize}
\item If $n'\leq_P n$, then $E(n)$ contains a copy of $E(n')$ as a subgraph.  The corresponding embedding is locally injective, so $E(n')\LocInHom E(n)$.

\item If $n\leq_P n'$, then Lemma~\ref{lem:sunlets} gives the explicit map $f:H(n)\LocInHom H(n')$.  For each $k<n$ with $k\leq_P n$, transitivity gives $k\leq_P n'$, so by induction there is a locally injective homomorphism $f_{k,n'}:E(k)\LocInHom E(n')$.  We choose these maps compatibly with the attachments.

If $k<n'$, then $E(n')$ contains the attached copy of $E(k)$ at $r_{3\cdot2^{k-1}}^{n'}$, and we take $f_{k,n'}$ to be that embedding.  Since the explicit map on the right cycle satisfies
\[
 f(r_{3\cdot2^{k-1}}^n)=r_{3\cdot2^{k-1}}^{n'},
\]
the identified pendant neighbors match.

If $k>n'$, then
\[
 f(r_{3\cdot2^{k-1}}^n)=r_0^{n'},
\]
because $3\cdot2^{k-1}$ is divisible by $R_{n'}=3\cdot2^{n'}$.  By the strengthened induction hypothesis, $f_{k,n'}$ maps $l_0^k$ to $l_{a(n')}^{n'}$, and therefore maps the pendant neighbor of $l_0^k$ to the middle vertex of the joining path between $l_{a(n')}^{n'}$ and $r_0^{n'}$.  This is also the image under $f$ of the pendant neighbor of $r_{3\cdot2^{k-1}}^n$.  Thus the maps agree on the identified vertex.  Combining $f$ with all maps $f_{k,n'}$ gives a locally injective homomorphism $E(n)\LocInHom E(n')$ with the required strengthened property.

\item Finally, suppose $n$ and $n'$ are incomparable in $(P,\leq_P)$.  We show that no locally injective homomorphism exists in either direction.  First consider a possible map $E(n')\LocInHom E(n)$.  The top left cycle of $H(n')$ would have to map to a left cycle of some copy $H(k)$ contained in $E(n)$.  If $k>n'$, then the target left cycle has a higher power of $2$ in its length and cannot divide the source length.  If $k<n'$, then Lemma~\ref{lem:sunlets} would imply $n'\leq_P k$; since every such $k$ occurring in $E(n)$ satisfies $k\leq_P n$, transitivity would give $n'\leq_P n$, a contradiction.  If $k=n'$, then the construction of $E(n)$ already implies $n'\leq_P n$, again a contradiction.  Hence $E(n')\LocInHom E(n)$ does not exist.

For the opposite direction, suppose $E(n)\LocInHom E(n')$.  The top left cycle of $H(n)$ must map to a left cycle of some $H(k)$ contained in $E(n')$.  Since every such $k$ satisfies $k\leq_P n'$ and $k<n$, Lemma~\ref{lem:sunlets} gives $n\leq_P k$.  Transitivity gives $n\leq_P n'$, a contradiction.  Hence no mapping $E(n)\LocInHom E(n')$ exists.
\end{itemize}
This completes the induction and proves that $E(x)\LocInHom E(y)$ holds exactly when $x\leq_P y$.  Therefore $(\ConnGraph,\LocInLeq)$ is universal.
\end{proof}

\begin{corollary}
\label{cor:locin-cactus-universal}
The partial order induced by locally injective homomorphisms on finite connected bipartite cactus graphs of maximum degree at most three is universal.  Consequently, universality already holds within finite connected planar graphs of treewidth at most two and maximum degree at most three.
\end{corollary}

\begin{proof}
In the proof of Theorem~\ref{thm:lochomouniv}, every representing graph $E(n)$ is obtained by joining cycles and pendant edges in a cactus-like way.  No two cycles share an edge, and every block is either an edge or a cycle; hence each $E(n)$ is a cactus.  The cycle lengths are
\[
p(n)2^n \quad\hbox{and}\quad 3\cdot 2^n,
\]
and are therefore even, so each $E(n)$ is bipartite.  Cycle vertices have degree at most three, the vertices created by identifying pendant neighbors have degree two, and all remaining pendant vertices have degree one; hence the maximum degree is at most three.  Cactus graphs are planar and have treewidth at most two.  Since the embedding constructed in Theorem~\ref{thm:lochomouniv} uses only these graphs, the restricted order is universal.
\end{proof}

\begin{prop}[Dualities]
For a finite set of undirected connected graphs $\mathcal{D}$ there is a finite set of undirected connected graphs
$\mathcal{F}$ such that $(\mathcal{F},\mathcal{D})$ is a generalized finite LI-duality pair
if and only if $\mathcal{D}$ consists of trees.  Moreover, if $n=\max_{D\in\mathcal{D}} |V_D|$, then $\mathcal{F}$ may be chosen to consist only of trees on at most $n+1$ vertices.
\end{prop}

\begin{proof}
Every duality pair must also be a duality pair when restricted to a given
equivalence class of degree refinement matrices.  By the same argument as in
Proposition~\ref{prop:locbidual} this is not possible for infinite equivalence classes.
On the other hand, finite equivalence classes correspond to acyclic connected graphs, that is, to trees.

Let $\mathcal{D}$ be a finite set of trees, and let $n$ be the maximum number of vertices of a graph in $\mathcal{D}$.
Let $\mathcal{F}$ be the set of all trees $T$ on at most $n+1$ vertices such that $T\LocInHom D$ for no $D\in\mathcal{D}$, and keep only the minimal members of this finite set.

Because the locally injective homomorphism order restricted to trees
is an embedding order, this set gives the required duality on trees.  If a graph $G$ contains a cycle, then it admits a locally injective homomorphism from a path on $n+1$ vertices.  This path belongs to $\mathcal{F}$, since no such path can be locally injectively mapped into a tree with at most $n$ vertices.  Thus the left-hand obstructions can all be chosen as trees on at most $n+1$ vertices.
\end{proof}

\section{Concluding remarks}

The comparison developed here shows that many constrained homomorphism orders are governed by simple one-sided finiteness principles.  Monomorphisms, embeddings and full homomorphisms are essentially past-finite after passing to their cores, while vertex-surjective, edge-surjective and surjective homomorphisms are future-finite.  These observations explain the existence of many finite dualities and also show why several of these orders cannot reproduce the full behavior of the ordinary homomorphism order.

The locally constrained orders are more delicate.  Locally surjective and locally bijective homomorphisms behave like covering maps on connected graphs and are future-finite inside the relevant degree-refinement classes.  Locally injective homomorphisms are closer to ordinary homomorphisms: all connected graphs are cores, the order is universal on connected graphs, and a density phenomenon holds across degree-refinement classes under natural hypotheses.

Several problems remain open.  The characterization of gaps in the locally injective order is only partial, and the constructions used here suggest that a complete description may depend on the interaction between degree refinement and covering-like behavior.  It would also be interesting to extend the unified approach to further variants of graph homomorphisms, including homomorphisms of relational structures, signed graphs, monounary algebras, line graph orders, minor-like graph orders and more recent quantum variants~\cite{NaserasrSenSopena2020SignedOrder,Jakubikova2021Monounary}.

\section*{Statements and Declarations}

\textbf{Funding}
Ji\v r\'\i{} Fiala was supported by M\v{S}MT \v{C}R grant LH12095 and GA\v{C}R grant P202/12/G061. Jan Hubi\v{c}ka was supported by grant ERC-CZ LL-1201 of the Czech Ministry of Education and CE-ITI P202/12/G061 of GA\v CR. Yangjing Long was supported by the Open Project Fund of Hubei Key Laboratory of Mathematical Physics.

\textbf{Competing interests}
The authors declare that they have no competing interests.

\textbf{Data availability}
No datasets were generated or analysed during the current study.

\textbf{Author contributions}
All authors contributed to the conception and development of the results, the writing of the manuscript, and the review of the final version.

\textbf{Use of AI tools}
During the preparation of this manuscript, the authors used AI-assisted editing tools for language polishing and structural organization. The authors reviewed and edited all content and take full responsibility for the final manuscript.

\begingroup
\setlength{\bibsep}{0pt plus 0.3ex}
\bibliography{constrainedhomo-order}

@InProceedings{Fiala2001,
  Title                    = {Complexity of Partial Covers of Graphs},
  Author                   = {Ji{\v{r}}\'{\i} Fiala and Jan Kratochv\'{\i}l},
  Booktitle                = {ISAAC},
  Year                     = {2001},
  Pages                    = {537--549},

  Bibsource                = {DBLP, http://dblp.uni-trier.de},
  Crossref                 = {DBLP:conf/isaac/2001},
  Ee                       = {http://dx.doi.org/10.1007/3-540-45678-3_46},
  doi = {10.1007/3-540-45678-3_46},
  Address = {Berlin, Heidelberg}
}

@InProceedings{Fiala2005,
  Title                    = {Matrix and Graph Orders Derived from Locally Constrained Graph Homomorphisms},
  Author                   = {Ji{\v{r}}\'{\i} Fiala and Dani{\"e}l Paulusma and Jan Arne Telle},
  Booktitle                = {MFCS},
  Year                     = {2005},
  Pages                    = {340--351},

  Bibsource                = {DBLP, http://dblp.uni-trier.de},
  Crossref                 = {DBLP:conf/mfcs/2005},
  Ee                       = {http://dx.doi.org/10.1007/11549345_30},
  doi = {10.1007/11549345_30},
  Address = {Berlin, Heidelberg}
}

@InProceedings{Kristiansen2000,
  Title                    = {Generalized {\it H}-Coloring of Graphs},
  Author                   = {Petter Kristiansen and Jan Arne Telle},
  Booktitle                = {ISAAC},
  Year                     = {2000},
  Pages                    = {456--466},

  Bibsource                = {DBLP, http://dblp.uni-trier.de},
  Crossref                 = {DBLP:conf/isaac/2000},
  Ee                       = {http://dx.doi.org/10.1007/3-540-40996-3_39},
  doi = {10.1007/3-540-40996-3_39},
  Address = {Berlin, Heidelberg}
}

@InCollection{Samal2013,
  Title                    = {Cycle-continuous mappings — order structure},
  Author                   = {{\v{S}}{\'{a}}mal, Robert},
  Booktitle                = {The Seventh European Conference on Combinatorics, Graph Theory and Applications},
  Publisher                = {Scuola Normale Superiore},
  Year                     = {2013},
  Editor                   = {Ne\v{s}et\v{r}il, Jaroslav and Pellegrini, Marco},
  Pages                    = {513--520},
  Series                   = {CRM Series},
  Volume                   = {16},
  doi = {10.1007/978-88-7642-475-5_81},
  ISBN                     = {978-88-7642-474-8},
  Language                 = {English},
  Url                      = {http://dx.doi.org/10.1007/978-88-7642-475-5_81},
  Address = {Pisa}
}

@Article{Ahlswede1995,
  Title                    = {A splitting property of maximal antichains},
  Author                   = {Ahlswede, Rudolf and Erd\H{o}s, P\'{e}ter L. and Graham, Niall},
  Journal                  = {Combinatorica},
  Year                     = {1995},
  Number                   = {4},
  Pages                    = {475--480},
  Volume                   = {15},
  doi = {10.1007/BF01192520},
  ISSN                     = {0209-9683},
  Language                 = {English},
  Publisher                = {Springer-Verlag},
  Url                      = {http://dx.doi.org/10.1007/BF01192520}
}

@Article{Ball2010,
  Title                    = {Dualities in full homomorphisms},
  Author                   = {Richard N. Ball and Jaroslav Ne\v{s}et\v{r}il and Ale\v{s} Pultr},
  Journal                  = {European Journal of Combinatorics},
  Year                     = {2010},
  Number                   = {1},
  Pages                    = {106--119},
  Volume                   = {31},

  Bibsource                = {DBLP, http://dblp.uni-trier.de},
  Ee                       = {http://dx.doi.org/10.1016/j.ejc.2009.04.004},
  doi = {10.1016/j.ejc.2009.04.004}
}

@Article{Ball2007,
  Title                    = {Finite Dualities, in Particular in Full Homomorphisms},
  Author                   = {Richard N. Ball and Jaroslav Ne\v{s}et\v{r}il and Ale\v{s} Pultr},
  Journal                  = {KAM-DIMATIA series},
  Year                     = {2007},
  Volume                   = {835}
}

@Article{Bass1990,
  Title                    = {Uniform tree lattices},
  Author                   = {Bass, Hyman and Kulkarni, Ravi},
  Journal                  = {Journal of the American Mathematical Society},
  Year                     = {1990},
  Number                   = {4},
  Pages                    = {843--902},
  Volume                   = {3},

  Publisher                = {JSTOR}
}

@Article{Bodirsky2012,
  Title                    = {The complexity of surjective homomorphism problems---a survey},
  Author                   = {Manuel Bodirsky and Jan K\'ara and Barnaby Martin},
  Journal                  = {Discrete Applied Mathematics},
  Year                     = {2012},
  Number                   = {12},
  Pages                    = {1680--1690},
  Volume                   = {160},
  doi = {10.1016/j.dam.2012.03.029},
  ISSN                     = {0166-218X},
  Url                      = {http://www.sciencedirect.com/science/article/pii/S0166218X12001333}
}

@article{Britz2001,
  title={Partially ordered sets},
  author={Britz, Thomas and Cameron, Peter},
  journal={J. of Formalized Mathematics},
  volume={1},
  year={2002}
}

@Article{Droste2005,
  Title                    = {Universal homogeneous causal sets},
  Author                   = {Droste, Manfred},
  Journal                  = {Journal of Mathematical Physics},
  Year                     = {2005},
  Number                   = {12},
  Pages                    = {10pp},
  Volume                   = {46},
  doi = {10.1063/1.2147607},
  Eid                      = {122503},
  Url                      = {http://scitation.aip.org/content/aip/journal/jmp/46/12/10.1063/1.2147607}
}

@Article{Feder2008,
  Title                    = {On realizations of point determining graphs, and obstructions to full homomorphisms},
  Author                   = {Tom{\'a}s Feder and Pavol Hell},
  Journal                  = {Discrete Mathematics},
  Year                     = {2008},
  Number                   = {9},
  Pages                    = {1639--1652},
  Volume                   = {308},

  Bibsource                = {DBLP, http://dblp.uni-trier.de},
  Ee                       = {http://dx.doi.org/10.1016/j.disc.2006.11.026},
  doi = {10.1016/j.disc.2006.11.026}
}

@Unpublished{Fiala,
  Title                    = {Constrained homomorphism orders},
  Author                   = {Fiala, Ji{\v{r}}\'{\i} and Hubi\v{c}ka, Jan and Long, Yangjing},
  Note                     = {In preparation}
}

@Article{Fialab,
  Title                    = {Universality of intervals of line graph order},
  Author                   = {Fiala, Ji{\v{r}}{\'\i} and Hubi{\v{c}}ka, Jan and Long, Yangjing},
  Journal                  = {European Journal of Combinatorics},
  Year                     = {2014},
  Pages                    = {221--231},
  Volume                   = {41},

  Publisher                = {Elsevier}
}

@Article{Fiala2008,
  Title                    = {Locally constrained graph homomorphisms --- structure, complexity, and applications},
  Author                   = {Ji{\v{r}}\'{\i} Fiala and Jan Kratochv\'{\i}l},
  Journal                  = {Computer Science Review},
  Year                     = {2008},
  Number                   = {2},
  Pages                    = {97--111},
  Volume                   = {2},

  Bibsource                = {DBLP, http://dblp.uni-trier.de},
  Ee                       = {http://dx.doi.org/10.1016/j.cosrev.2008.06.001},
  doi = {10.1016/j.cosrev.2008.06.001}
}

@Article{Fiala2006,
  Title                    = {Cantor-{B}ernstein type theorem for locally constrained graph homomorphisms},
  Author                   = {Ji{\v{r}}\'{\i} Fiala and Jana Maxov{\'a}},
  Journal                  = {European Journal of Combinatorics},
  Year                     = {2006},
  Number                   = {7},
  Pages                    = {1111--1116},
  Volume                   = {27},

  Bibsource                = {DBLP, http://dblp.uni-trier.de},
  Ee                       = {http://dx.doi.org/10.1016/j.ejc.2006.06.003},
  doi = {10.1016/j.ejc.2006.06.003}
}

@article{FialaFractal,
  title={Fractal property of the graph homomorphism order},
  author={Fiala, Ji{\v{r}}{\'\i} and Hubi{\v{c}}ka, Jan and Long, Yangjing and Ne{\v{s}}et{\v{r}}il, Jaroslav},
  journal={European Journal of Combinatorics},
  volume={66},
  pages={101--109},
  year={2017},
  publisher={Elsevier}
}

@Article{Foniok2008,
  Title                    = {Generalised dualities and maximal finite antichains in the homomorphism order of relational structures},
  Author                   = {Jan Foniok and Jaroslav Ne{\v{s}}et{\v{r}}il and Claude Tardif},
  Journal                  = {European Journal of Combinatorics},
  Year                     = {2008},
  Number                   = {4},
  Pages                    = {881--899},
  Volume                   = {29},

  Bibsource                = {DBLP, http://dblp.uni-trier.de},
  Ee                       = {http://dx.doi.org/10.1016/j.ejc.2007.11.017},
  doi = {10.1016/j.ejc.2007.11.017}
}

@Article{Hedrlin1969,
  Title                    = {On universal partly ordered sets and classes.},
  Author                   = {Zen\v{e}k Hedrl\'\i{}n},
  Journal                  = {J. Algebra},
  Year                     = {1969},
  Pages                    = {503--509},
  Volume                   = {11(4)}
}

@article{hell2013,
  title={Point determining digraphs,$\{$0, 1$\}$-matrix partitions, and dualities in full homomorphisms},
  author={Hell, Pavol and Hern{\'a}ndez-Cruz, C{\'e}sar},
  journal={Discrete Mathematics},
  volume={338},
  number={10},
  pages={1755--1762},
  year={2015},
  publisher={Elsevier}
}

@Book{Hell2004,
  Title                    = {Graphs and homomorphisms},
  Author                   = {Pavol Hell and Jaroslav Ne\v{s}et{\v{r}}il},
  Publisher                = {Oxford University Press},
  Year                     = {2004},
  Series                   = {Oxford lecture series in mathematics and its applications},

  ISBN                     = {9780198528173},
  Lccn                     = {2004533593},
  Url                      = {http://books.google.de/books?id=bJXWV-qK7kYC},
  Address = {Oxford}
}

@InCollection{Hubicka2011,
  Booktitle                    = {Model Theoretic Methods in Finite Combinatorics},
  Author                   = {Jan Hubi{\v{c}}ka and Jaroslav Ne{\v{s}}et{\v{r}}il},
  Title                  = {Some examples of universal and generic partial orders},
  Editor                   = {Martin Grohe and Johann A. Makowsky},
  Pages                    = {293--318},
  Publisher                = {American Mathematical Society},
  Year                     = {2011},
  Address = {Providence, RI}
}

@Article{Hubicka2005,
  Title                    = {Universal partial order represented by means of oriented trees and other simple graphs},
  Author                   = {Jan Hubi\v{c}ka and Jaroslav Ne\v{s}et\v{r}il},
  Journal                  = {European Journal of Combinatorics},
  Year                     = {2005},
  Number                   = {5},
  Pages                    = {765--778},
  Volume                   = {26},

  Bibsource                = {DBLP, http://dblp.uni-trier.de},
  Ee                       = {http://dx.doi.org/10.1016/j.ejc.2004.01.008},
  doi = {10.1016/j.ejc.2004.01.008}
}

@Article{Hubicka2004,
  Title                    = {Finite Paths are Universal},
  Author                   = {Jan Hubi\v{c}ka and Jaroslav Ne\v{s}et\v{r}il},
  Journal                  = {Order},
  Year                     = {2004},
  Number                   = {3},
  Pages                    = {181--200},
  Volume                   = {21},

  Bibsource                = {DBLP, http://dblp.uni-trier.de},
  Ee                       = {http://dx.doi.org/10.1007/s11083-004-3345-9},
  doi = {10.1007/s11083-004-3345-9}
}

@Article{Kwuida2011,
  Title                    = {On the Homomorphism Order of Labeled Posets},
  Author                   = {Kwuida, Léonard and Lehtonen, Erkko},
  Journal                  = {Order},
  Year                     = {2011},
  Number                   = {2},
  Pages                    = {251--265},
  Volume                   = {28},
  doi = {10.1007/s11083-010-9169-x},
  ISSN                     = {0167-8094},
  Language                 = {English},
  Publisher                = {Springer Netherlands},
  Url                      = {http://dx.doi.org/10.1007/s11083-010-9169-x}
}

@Article{Lehtonen2008,
  Title                    = {Labeled posets are universal},
  Author                   = {Erkko Lehtonen},
  Journal                  = {European Journal of Combinatorics},
  Year                     = {2008},
  Number                   = {2},
  Pages                    = {493--506},
  Volume                   = {29},

  Bibsource                = {DBLP, http://dblp.uni-trier.de},
  Ee                       = {http://dx.doi.org/10.1016/j.ejc.2007.02.005},
  doi = {10.1016/j.ejc.2007.02.005}
}

@Article{Lehtonen2010,
  Title                    = {Minors of Boolean functions with respect to clique functions and hypergraph homomorphisms},
  Author                   = {Lehtonen, Erkko and Ne{\v{s}}et{\v{r}}il, Jaroslav},
  Journal                  = {European Journal of Combinatorics},
  Year                     = {2010},
  Number                   = {8},
  Pages                    = {1981--1995},
  Volume                   = {31},

  Publisher                = {Elsevier}
}

@Article{Leighton1982,
  Title                    = {Finite common coverings of graphs},
  Author                   = {Frank Thomson Leighton},
  Journal                  = {Journal of Combinatorial Theory, Series B},
  Year                     = {1982},
  Number                   = {3},
  Pages                    = {231--238},
  Volume                   = {33},
  doi = {10.1016/0095-8956(82)90042-9},
  ISSN                     = {0095-8956},
  Url                      = {http://www.sciencedirect.com/science/article/pii/0095895682900429}
}

@Article{Nesetril2000,
  Title                    = {Duality Theorems for Finite Structures (Characterising Gaps and Good Characterisations).},
  Author                   = {Jaroslav Ne{\v{s}}et{\v{r}}il and Claude Tardif},
  Journal                  = {Journal of Combinatorial Theory, Series B},
  Year                     = {2000},
  Number                   = {1},
  Pages                    = {80--97},
  Volume                   = {80},
  Date                     = {2005-02-15},
  Description              = {dblp},
  Ee                       = {http://dx.doi.org/10.1006/jctb.2000.1970},
  Url                      = {http://dblp.uni-trier.de/db/journals/jct/jctb80.html#NesetrilT00},
  doi = {10.1006/jctb.2000.1970}
}

@Article{Nesetril1971,
  Title                    = {Homomorphisms of derivative graphs},
  Author                   = {Jaroslav Ne\v{s}et{\v{r}}il},
  Journal                  = {Discrete Mathematics},
  Year                     = {1971},
  Number                   = {3},
  Pages                    = {257--268},
  Volume                   = {1}
}

@Article{Nesetril2006,
  Title                    = {Minimal universal and dense minor closed classes},
  Author                   = {Jaroslav Ne\v{s}et{\v{r}}il and Yared Nigussie},
  Journal                  = {European Journal of Combinatorics},
  Year                     = {2006},
  Number                   = {7},
  Pages                    = {1159--1171},
  Volume                   = {27},
  Date                     = {2007-01-23},
  Description              = {dblp},
  Ee                       = {http://dx.doi.org/10.1016/j.ejc.2006.06.012},
  Url                      = {http://dblp.uni-trier.de/db/journals/ejc/ejc27.html#NesetrilN06},
  doi = {10.1016/j.ejc.2006.06.012}
}

@Article{Neumann2011,
  Title                    = {On {L}eighton's graph covering theorem},
  Author                   = {Walter D. Neumann},
  Journal                  = {Groups, Geometry, and Dynamics},
  Year                     = {2011},
  Pages                    = {863--872},
  Volume                   = {4}
}

@Book{pultr1980combinatorial,
  Title                    = {Combinatorial, algebraic, and topological representations of groups, semigroups, and categories},
  Author                   = {Pultr, Ale\v{s} and Trnkov{\'a}, V\v{e}ra},
  Publisher                = {North-Holland Publishing Company},
  Year                     = {1980},
  Series                   = {North-Holland mathematical library},

  ISBN                     = {9780444850836},
  Lccn                     = {80499856},
  Url                      = {http://books.google.de/books?id=uyDMgbpXGY8C},
  Address = {Amsterdam}
}

@Article{Sumner1973,
  Title                    = {Point determination in graphs},
  Author                   = {David P. Sumner},
  Journal                  = {Discrete Mathematics},
  Year                     = {1973},
  Pages                    = {179--187},
  Volume                   = {5},
  Part                     = {2}
}

@Article{Welzl1982,
  Title                    = {Color-families are dense},
  Author                   = {Emo Welzl},
  Journal                  = {Theoretical Computer Science},
  Year                     = {1982},
  Number                   = {1},
  Pages                    = {29--41},
  Volume                   = {17},
  doi = {10.1016/0304-3975(82)90129-3},
  ISSN                     = {0304-3975},
  Url                      = {http://www.sciencedirect.com/science/article/pii/0304397582901293}
}

@MastersThesis{Xie2006,
  Title                    = {Obstructions to Trigraph Homomorphisms},
  Author                   = {Wing Xie},
  School                   = {Simon Fraser University},
  Year                     = {2006}
}

@Proceedings{DBLP:conf/isaac/2001,
  Title                    = {Algorithms and Computation, 12th International Symposium, ISAAC 2001, Christchurch, New Zealand, December 19-21, 2001, Proceedings},
  Year                     = {2001},
  Editor                   = {Peter Eades and Tadao Takaoka},
  Publisher                = {Springer},
  Series                   = {Lecture Notes in Computer Science},
  Volume                   = {2223},

  Bibsource                = {DBLP, http://dblp.uni-trier.de},
  Booktitle                = {ISAAC},
  ISBN                     = {3-540-42985-9}
}

@Proceedings{DBLP:conf/mfcs/2005,
  Title                    = {Mathematical Foundations of Computer Science 2005, 30th International Symposium, MFCS 2005, Gdansk, Poland, August 29 - September 2, 2005, Proceedings},
  Year                     = {2005},
  Editor                   = {Joanna Jedrzejowicz and Andrzej Szepietowski},
  Publisher                = {Springer},
  Series                   = {Lecture Notes in Computer Science},
  Volume                   = {3618},

  Bibsource                = {DBLP, http://dblp.uni-trier.de},
  Booktitle                = {MFCS},
  ISBN                     = {3-540-28702-7}
}

@Proceedings{DBLP:conf/isaac/2000,
  Title                    = {Algorithms and Computation, 11th International Conference, ISAAC 2000, Taipei, Taiwan, December 18-20, 2000, Proceedings},
  Year                     = {2000},
  Editor                   = {D. T. Lee and Shang-Hua Teng},
  Publisher                = {Springer},
  Series                   = {Lecture Notes in Computer Science},
  Volume                   = {1969},

  Bibsource                = {DBLP, http://dblp.uni-trier.de},
  Booktitle                = {ISAAC},
  ISBN                     = {3-540-41255-7}
}

@Book{k:Biggs74,
  author    = {Norman Biggs},
  title	    = {Algebraic Graph Theory},
  year	    = {1974},
  publisher = {Cambridge University Press},
  Address = {Cambridge}
}

@Book{k:Reidemeister32,
  author    = {Reidemeister, Kurt},
  title     = {Einf{\"u}hrung in die kombinatorische Topologie},
  language  = {German},
  publisher = {Braunschweig: Friedr. Vieweg\&Sohn A.-G.},
  year	    = {1932},
  Address = {Braunschweig}
}

@Article{n:Angluin80,
  author    = {Dana Angluin},
  title	    = {Local and global properties in networks of processors},
  journal   = {Proceedings of the 12th ACM Symposium on Theory of Computing},
  year	    = {1980},
  pages	    = {82--93}
}

@Article{n:MNS02,
  author    = {Malni{\v{c}}, Aleksander and Nedela, Roman and {\v{S}}koviera,
	      Martin},
  title	    = {Regular homomorphisms and regular maps},
  journal   = {European J. Combin.},
  fjournal  = {European Journal of Combinatorics},
  volume    = {23},
  year	    = {2002},
  number    = {4},
  pages	    = {449--461},
  issn	    = {0195-6698}
}

@article{Fiala2015Universality,
  author = {Fiala, Ji{\v{r}}{\'\i} and Hubi{\v{c}}ka, Jan and Long, Yangjing},
  title = {An universality argument for graph homomorphisms},
  journal = {Electronic Notes in Discrete Mathematics},
  volume = {49},
  pages = {643--649},
  year = {2015},
  doi = {10.1016/j.endm.2015.06.087},}

@article{Fiala2017GapsFull,
  author = {Fiala, Ji{\v{r}}{\'\i} and Hubi{\v{c}}ka, Jan and Long, Yangjing},
  title = {Gaps in full homomorphism order},
  journal = {Electronic Notes in Discrete Mathematics},
  volume = {61},
  pages = {429--435},
  year = {2017},
  doi = {10.1016/j.endm.2017.06.070},}

@article{Bulteau2024Algorithmic,
  author = {Laurent Bulteau and Konrad K. Dabrowski and Noleen K{\"o}hler and Sebastian Ordyniak and Dani{\"e}l Paulusma},
  title = {An Algorithmic Framework for Locally Constrained Homomorphisms},
  journal = {SIAM Journal on Discrete Mathematics},
  volume = {38},
  number = {2},
  pages = {1315--1350},
  year = {2024},
  doi = {10.1137/22M1513290}
}

@inproceedings{Dvorak2022ListLS,
  author = {Pavel Dvo{\v{r}}{\'a}k and Tom{\'a}{\v{s}} Masa{\v{r}}{\'i}k and Jana Novotn{\'a} and Monika Krawczyk and Pawe{\l} Rz{\k{a}}{\.z}ewski and Aneta {\.Z}uk},
  title = {{List Locally Surjective Homomorphisms in Hereditary Graph Classes}},
  booktitle = {33rd International Symposium on Algorithms and Computation (ISAAC 2022)},
  series = {Leibniz International Proceedings in Informatics (LIPIcs)},
  volume = {248},
  pages = {30},
  note = {Article 30, 15 pp.},
  publisher = {Schloss Dagstuhl -- Leibniz-Zentrum f{\"u}r Informatik},
  address = {Dagstuhl, Germany},
  year = {2022},
  doi = {10.4230/LIPIcs.ISAAC.2022.30}
}

@article{GuzmanPro2023FullPathsCycles,
  author = {Santiago Guzm{\'a}n-Pro},
  title = {Full-homomorphisms to paths and cycles},
  journal = {arXiv preprint arXiv:2212.13313},
  year = {2023},
  doi = {10.48550/arXiv.2212.13313}
}

@article{Larose2019SurjectiveReflexive,
  author = {Beno{\^\i}t Larose and Barnaby Martin and Dani{\"e}l Paulusma},
  title = {Surjective {$H$}-Colouring over Reflexive Digraphs},
  journal = {ACM Transactions on Computation Theory},
  volume = {11},
  number = {1},
  pages = {3},
  note = {Article 3, 21 pp.},
  year = {2019},
  doi = {10.1145/3282431}
}

@article{Focke2019CountingSurjective,
  author = {Jacob Focke and Leslie Ann Goldberg and Stanislav {\v{Z}}ivn{\'y}},
  title = {The Complexity of Counting Surjective Homomorphisms and Compactions},
  journal = {SIAM Journal on Discrete Mathematics},
  volume = {33},
  number = {2},
  pages = {1006--1043},
  year = {2019},
  doi = {10.1137/17M1153182}
}

@article{Chen2024SurjectiveCSP,
  author = {Hubie Chen},
  title = {Algebraic Global Gadgetry for Surjective Constraint Satisfaction},
  journal = {Computational Complexity},
  volume = {33},
  number = {1},
  pages = {7},
  year = {2024},
  doi = {10.1007/s00037-024-00253-4}
}

@article{KratochvilNedela2025CoversSemicovers,
  author = {Jan Kratochv{\'\i}l and Roman Nedela},
  title = {Graph covers and semi-covers: Who is stronger?},
  journal = {arXiv preprint arXiv:2504.17387},
  year = {2025},
  doi = {10.48550/arXiv.2504.17387}
}

@article{NaserasrSenSopena2020SignedOrder,
  author = {Reza Naserasr and Sagnik Sen and {\'E}rik Sopena},
  title = {The homomorphism order of signed graphs},
  journal = {Journal of Combinatorial Mathematics and Combinatorial Computing},
  volume = {116},
  pages = {169--182},
  year = {2020}
}

@article{Jakubikova2021Monounary,
  author = {Danica Jakub{\'i}kov{\'a}-Studenovsk{\'a}},
  title = {Homomorphism Order of Connected Monounary Algebras},
  journal = {Order},
  volume = {38},
  pages = {257--269},
  year = {2021},
  doi = {10.1007/s11083-020-09539-y}
}
\endgroup

\end{document}